\documentclass[11pt]{article}
\usepackage[margin=1.2in]{geometry}
\usepackage{amsfonts,amsmath,amssymb,amsthm,bm}
\usepackage{pifont}
\usepackage{pgfplots}
\usepackage{subcaption}
\usepackage{mathtools}
\usepackage{graphicx}
\usepackage{authblk}
\usepackage{algorithm}
\usepackage{algpseudocode}
\usepackage{tikz}
\usepackage[
colorlinks,
citecolor=blue,
urlcolor=black,
linkcolor=blue, 
backref=page,
linktocpage=true]{hyperref}
\allowdisplaybreaks 
\numberwithin{equation}{section}

\newtheorem{theorem}{Theorem}[section]
\newtheorem{lemma}[theorem]{Lemma}

\theoremstyle{definition}
\newtheorem{assumption}{Assumption}[section]

\newtheorem{example}{Example}[section]
\theoremstyle{remark}
\newtheorem{remark}{Remark}[section]

\author[1]{Topi Halme}
\author[2]{H. Vincent Poor}
\author[1]{Visa Koivunen}
\affil[1]{Department of Information and Communications Engineering, Aalto University, Finland}
\affil[1]{\texttt{\{topi.halme, visa.koivunen\}@aalto.fi}}
\affil[2]{Electrical and Computer Engineering, Princeton University, USA}

\def\ARL{{\textsc{ARL}}}

\def\Wmax{{W_\text{max}}}

\def\bZ{{\mathbf{Z}}}

\def\bX{{\mathbf{X}}}

\def\calP{{\mathcal{P}}}

\def\bXw{{\bX_{n-1}^{(w)}}}

\def\calN{{\cal N}}
\def\calR{{\cal R}}
\def\calF{{\cal F}}
\def\calC{{\cal C}}
\def\calD{{\cal D}}

\def\Ex{{\mathbb{E}}}

\def\btheta{{\boldsymbol{\theta}}}

\def\lhat{{\hat{\ell}}}

\def\bthetahat{{\skew{2}{\widehat}{\btheta}}}

\def\phat{{\hat{f}}}
\def\phat{{\hat{p}}}

\def\bTheta{{\boldsymbol{\Theta} }}

\def\bbR{{\mathbb{R}}}

\def\calW{{\mathcal{W}}}

\def\bX{{\mathbf{X}}}

\def\bZ{{\mathbf{Z}}}

\newcommand{\ind}[1]{\textbf{1}{\left\{#1\right\}}}

\newcommand\norm[1]{\lVert#1\rVert}
\newcommand\delay[1]{\calD_P\!\left(#1\right)}

\newcommand{\KL}[2]{D(#1 \,\|\, #2)}
\def\define{\coloneq}

\title{Sequential Change Detection using %
 Mixtures of Predictive Distributions}

\graphicspath{{figs/}}
\begin{document}

\maketitle

\begin{abstract}
This paper studies the problem of detecting a change in the distribution of a sequence of independent observations when the post-change distribution is unknown. We propose a novel change detection algorithm, termed Predictive-Mixture CuSum (PM-CuSum), which combines predictive distributions constructed from sliding windows of different lengths within a CuSum recursion. The predictive distributions are aggregated using adaptive weights based on their recent predictive performance. We show that PM-CuSum achieves first-order asymptotic optimality under mild conditions, and that its asymptotic delay bound has a smaller remainder order than what is achieved procedures using a single fixed (even oracle) window. Numerical simulations demonstrate that PM-CuSum performs well compared to existing methods. Moreover, it is demonstrated that forming likelihood ratios using full predictive distributions can substantially improve performance compared to plug-in likelihoods.
\end{abstract}

\section{Introduction}

In the quickest change detection (QCD) problem, the goal is to detect a change in the underlying probability model of a sequence of observations as quickly as possible while controlling a suitable rate of false alarms \cite{TARTAKOVSKY_BOOK,POOR_BOOK}. This problem finds applications in a wide variety of fields, including quality control, sensor networks, environmental monitoring, and cybersecurity. In the classical QCD problem, the pre- and post-change distributions are assumed to be completely known. In this case, the Cumulative Sum (CuSum) procedure introduced by Page \cite{PAGE_1954} exactly minimizes Lorden's worst-case average detection delay subject to a false alarm rate constraint \cite{LORDEN_1971,MOUSTAKIDES_1986}. The Shiryaev-Roberts (SR) procedure \cite{Roberts1966} also has strong optimality properties with respect to various metrics \cite{Polunchenko}.

In most applications, the post-change distribution is not known in advance as it depends on the event causing the change. For example, in a sensor network, the same event may produce different signal intensities at different sensors depending on the event location. When the post-change distribution depends on an unknown (finite-dimensional) parameter, the standard approaches in the literature are generalized likelihood ratio (GLR) \cite{LORDEN_1971,LAI_1998} or mixture-based tests \cite{LAI_1998}. In the former class of procedures, the unknown parameter is essentially replaced by its maximum likelihood estimate, while the latter integrates the parameter over a prior distribution. Although these procedures are effective and conceptually clean, they can be computationally expensive because they generally do not admit a simple recursive update. A recursive update is among the usual desiderata for QCD procedures, since observations can arrive at high rates and decisions must be made online.

Many approaches have been proposed in the literature to lessen the computational burden, see e.g. \cite{wang_review} for a recent review. In some special cases, such as for the detection of a Gaussian mean-shift \cite{Romano2023} or Poisson rate parameter \cite{Ward2025, DeLucia22042026}, the computation of the GLR test statistic can be made efficient. An alternative class of methods utilizes an adaptive estimate of the post-change parameter based on past data to approximate the post-change distribution, and plugs the estimate into the recursive CuSum or SR update. In the sequential analysis literature, the idea appeared already in the work of Robbins and Siegmund \cite{ROBBINS_1974,ROBBINS_1972}, but was introduced to change detection later by Sparks \cite{SPARKS_2000} and Lorden and Pollak \cite{LORDEN_2005}. A computationally efficient adaptive procedure is presented in \cite{LORDEN_2008}, where the estimator is updated recursively but restarted whenever the detection statistic hits zero. Related ideas of restarting were also proposed in \cite{TARTAKOVSKY_2006}. In \cite{CAO_2018}, the estimator is learned using online mirror descent. Recently \cite{XIE_2023} proposed estimating the post-change parameter from a sliding window of fixed size $w$. It was shown that window-limited-CuSum (WL-CuSum) achieves asymptotic optimality with a window size that grows much slower than the window size required in the window-limited GLR test. 

The two central research questions and related design choices for an adaptive QCD procedure are:
\begin{enumerate}
    \item[1)] how much past data to use when estimating or predicting the post-change model? and
    \item[2)] how can that data along with the new observation be converted into a new test-statistic increment?
\end{enumerate}
The typical approaches to item 1) are to either use a fixed window length as in \cite{XIE_2023}, or to keep increasing the window unless the test statistic hits some small value, at which point the window length is reset to one as in \cite{LORDEN_2008,TARTAKOVSKY_2006}. When using a fixed window, the choice of the window size is crucial for the performance of the procedure, and may not always be easy to choose beforehand. Similarly, for the resetting approach, the choice of the reset threshold is important. Regarding item 2), the existing methods have all formed the new likelihood ratio by directly plugging in the parameter estimated using some method, such as maximum likelihood or method of moments \cite{LORDEN_2005,XIE_2023,CAO_2018}.

In this paper, we propose an algorithm which, to the best of our knowledge, is novel in the QCD literature with respect to both points. For item 1), we use a set of window sizes of different lengths, and form the new likelihood ratio as an adaptive mixture of the likelihood ratios corresponding to each window size. The window lengths that have recently achieved the best predictive performance are given the most weight. This way, the algorithm automatically learns the appropriate window size for the problem at hand, reducing the need for tuning. For item 2), instead of forming the new likelihood ratio by plugging in a point estimate of the parameter, we consider full predictive distributions for the new observation based on the past data, and use this distribution to compute the new likelihood ratio. Results from the predictive density estimation literature \cite{brown2008admissible} show that such predictive distributions can dominate plug-in rules under Kullback-Leibler loss in standard models. To the best of our knowledge, this predictive-density approach to constructing test statistic increments has not been previously considered in the QCD literature.

\subsubsection*{Contributions} The main contributions of this paper are the following:
\begin{itemize}
    \item We propose an adaptive QCD algorithm named Predictive-Mixture CuSum (PM-CuSum) that employs a set of candidate predictors based on different window sizes, and combines them using adaptive weights.
    \item We prove that the algorithm is asymptotically first-order optimal in the sense of Lorden's worst-case delay under general conditions, including nonparametric classes of post-change distributions. Moreover, the procedure has a smaller asymptotic remainder term than the popular WL-CuSum procedure \cite{XIE_2023} or its parallel variant. 
    \item The procedure is broadly applicable to both parametric and nonparametric approaches for constructing the predictive distribution, and we provide concrete constructions for some parametric models. 
    \item Numerical simulations demonstrate that the proposed algorithm performs well compared to existing methods, while being computationally efficient. It is also observed that the performance gain obtained by using a full predictive density in place of a plug-in distribution can be significant, even in simple Gaussian settings.
\end{itemize}

The rest of this paper is structured as follows. In Section~\ref{sec:problem_formulation}, we formalize the problem and review the existing approaches. In Section~\ref{sec:pred_cusum}, we describe the proposed algorithm, and analyze its asymptotic properties. In Section~\ref{sec:simulations}, we present numerical simulations comparing the performance of the proposed algorithm to existing methods. Section~\ref{sec:conclusion} concludes the paper and discusses future research.

\subsubsection*{Notation}
We use $\bX_{n-1}^{(w)}$ to denote $(X_{n-w},\dots,X_{n-1})$, i.e. the past $w$ samples observed prior to time $n$. The full history $\bX_{n-1}^{(n-1)}= (X_1,\dots,X_{n-1})$ is denoted by $\bX_{n-1}$ for short. If $X_n$ is a vector, we use $X_n^{(j)}$ to denote the $j$-th component of $X_n$. We write $\calF_n = \sigma(\bX_{n})$ to denote the $\sigma$-algebra generated by the observations up to time step $n$. The expectation operator when change happens at time $\nu$ and the true distribution is $P$ is denoted by $\Ex_\nu^P$, and $\Ex_\infty$ is the expectation when no change happens. 

\section{Problem formulation}\label{sec:problem_formulation}
Let $X_1, X_2, \ldots$ be an independent sequence of random variables observed sequentially. Suppose there exists an unknown deterministic change-point $\nu$ such that
\begin{equation}
  X_n \overset{\mbox{i.i.d}}{\sim} \begin{cases}
~~Q, \quad &n = 1,2,...,\nu -1  \\
~~P, ~~~ &n = \nu, \nu+1,... 
\end{cases}
\end{equation}
i.e. at time $\nu$ the pdf of the observations changes from a known distribution $Q$ to another distribution $P$. It is assumed that $Q$ and $P$ have density functions $q$ and $p$, respectively, with respect to a common dominating measure $\mu$. It is assumed that $Q$ is completely known, while the post-change distribution is unknown. We assume that $P \in \mathcal{P}$ belongs to some known family of distributions. For example, $\calP$ can be a family of parametric distributions parametrized by some vector $\btheta \in \bbR^k$ of length $k$. The goal is to detect the change from $Q$ to $P$ as quickly as possible, while avoiding false alarms. In this setting, a change detection procedure is a stopping time $T$ with respect to the filtration $\{\calF_n\}$, where $\calF_n = \sigma(X_1,\dots,X_n)$ is the $\sigma$-algebra generated by the first $n$ observations. The procedure is evaluated with respect to the following two criteria:

\noindent\textbf{False Alarm Criterion.} The rate of false alarms for a procedure $T$ is quantified by the average run length (ARL) to false alarm in the pre-change regime. The family of tests with ARL at least $\gamma$ is denoted by $\calC_\gamma$, i.e. 
 \begin{equation}
     \calC_\gamma \define \left\{T : \Ex_\infty(T) \geq \gamma\right\}.
 \end{equation}

\noindent\textbf{Delay Criterion.}
The detection delay is measured in the "worst worst-case" as defined by Lorden  \cite{LORDEN_1971}. For $P \in \calP$ 
\begin{equation}\label{lorden}
\delay T \define \underset{\nu\geq 1}{\sup}\text{ ess} \sup \Ex^P_\nu((T-\nu+1)^+ | \mathcal{F}_{\nu-1}),
\end{equation}
with the worst-case evaluated over both all possible change-points $\nu$ and pre-change observations.

If the post-change distribution is known, i.e. $\calP = \{P\}$ is a singleton, QCD theory is well-developed, and some optimal tests are known. Lorden's delay is minimized by the CuSum test \cite{PAGE_1954, MOUSTAKIDES_1986}, which is also computationally convenient. The test is based on thresholding the recursive statistic
\begin{equation}\label{eq:cusum_stat}
    W_n 
    = \max(W_{n-1},0) + \log \frac{p(X_n)}{q(X_n)}, \quad W_0 = 0.
\end{equation}
When $\calP$ is a parametric class depending on an unknown $\btheta$, the post-change density $p$ can be written as $p(\cdot; \btheta)$ to highlight the dependence on $\btheta$.
Adaptive change-point detection methods replace the unknown $\btheta$ by a \emph{predictable}\footnote{A sequence of random variables $\bthetahat_1,\bthetahat_2,...$ is said to be predictable (with respect to a filtration $\{\calF_n\}$) if $\bthetahat_n \in \calF_{n-1}$, i.e. the value of $\bthetahat_n$ is determined by the information available up to time $n-1$.} estimator $\bthetahat_n \in \calF_{n-1}$ which is then inserted as the parameter in the CuSum-update
\begin{equation}\label{eq:general_adaptive1}
    W_n = \max(W_{n-1},0) + \log \frac{p(X_n; \bthetahat_n)}{q(X_n)}.
\end{equation}
Various ways of constructing the estimator $\bthetahat$ have been proposed in the literature. In particular, an important question is identifying the appropriate block of past data from which to estimate the parameter at any given time. Having a very long memory will lead to pre-change data biasing the estimation, while a memory that is too short will yield noisy estimates.

In \cite{LORDEN_2008}, the authors propose to restart the estimator $\bthetahat$ whenever the detection statistic hits zero. In \cite{CAO_2018}, the estimator is learned using online mirror descent. Recently \cite{XIE_2023} proposed estimating $\btheta$ from a sliding window of fixed size $w$. It was shown that asymptotic optimality can be achieved with a window size that grows much slower than the window size required in the window-limited GLR test. After the appropriate estimation window has been selected, relatively few works have studied how to form the new likelihood ratio using the chosen data. The standard approach in the above works has been to plug in the maximum likelihood estimator of $\theta$. The use of shrinkage estimation for $\bthetahat$ was first proposed in \cite{WANG_2015} and recently shown to be a uniform improvement over the MLE for detecting Gaussian mean shifts in \cite{halme2025quickest}.

\subsection{Change detection by prediction}\label{subsec:cd_by_prediction}

More generally, the adaptive change detection tests of the form \eqref{eq:general_adaptive1} may be seen as consecutive predictions of the next observation $X_n$. At time $n$, all available history $\bX_{n-1}$ may be used to formulate a predictive distribution $\phat_n(x)$ that assigns a probability (density) to each possible outcome $x$. The ``quality'' of the prediction is measured by the realized log-likelihood ratio $\lhat_n = \log \phat_n(X_n)/q(X_n)$, where large values indicate that $\phat_n$ predicted the new observation better than the pre-change model $q$. 
The log-likelihood ratio is used to recursively update a test-statistic
\begin{equation}\label{eq:general_adaptive}
    S_n = \max(S_{n-1},0) + \lhat_n = \max(S_{n-1},0) + \log \frac{\phat_n(X_n)}{q(X_n)}
    \end{equation}
In the pre-change regime where $X_n \sim q$, when $\phat_n$ is any probability density function independent of $X_n$, the likelihood update $e^{\lhat_n}$ always has unit expectation, since
\begin{equation}
    \Ex_\infty e^{\lhat_n} = \Ex_\infty\!\!\left(\frac{\phat_n(X_n)}{q(X_n)}\right)=\int \frac{\phat_n(x)}{q(x)}q(x)d x= 1.
\end{equation}
Therefore, being able to consistently obtain large likelihood ratios, and consequently large values of $S_n$, with some sequence of predictors $\phat_n$ implies that the pre-change model is no longer in place. The general approach can be seen to fall under the ``testing by betting'' \cite{shafer_betting} framework that has seen significant interest in recent years. The process $\{\exp(S_n)\}$ can also be seen as an \emph{e-detector}, as defined in \cite{shin_ramdas_rinaldo_2024}. Of particular interest is the expectation of $\lhat_n$ after the change $n\geq\nu$, which defines the rate of growth in $\{S_n\}$ and can be written as 
    \begin{equation}\label{eq:drift_decomposition}
    \Ex_\nu^P(\lhat_n)= 
        \Ex_\nu^P\left(\log \frac{\phat_n(X_n)}{q(X_n)}\right) = 
        \Ex_\nu^P\left(\log \frac{p(X_n)}{q(X_n)}\right) - \Ex_\nu^P\left(\log \frac{p(X_n)}{\phat_n(X_n)}\right).
    \end{equation}
    On the right hand side of~\eqref{eq:drift_decomposition}, the first term is simply the Kullback-Leibler (KL) divergence between the pre- and post-change distributions
\begin{equation}
    \Ex_\nu^P\left(\log \frac{p(X_n)}{q(X_n)}\right)= \KL{p}{q} \define \int\log\left(\frac{p(x)}{q(x)}\right)p(x)dx.
\end{equation}
The second term is the expected KL-divergence 
\begin{equation}
    \mathcal{R}_\nu(\phat_n)\define \Ex_\nu^P[\KL{p}{\phat_n}] = \Ex_\nu^P\left(\log \frac{p(X_n)}{\phat_n(X_n)}\right)
\end{equation}
between the true post-change distribution $p$ and the predicted $\phat_n$, where the expectation is taken over the randomness in both $\phat_n$ and $X_n$. It follows from standard properties of KL-divergence that $\mathcal{R}_\nu(\phat_n)$ is non-negative for $n \geq \nu$ and equal to zero if and only if $\phat_n = p$ almost surely, maximizing $\Ex_1^P(\lhat_n)$. Obviously, inserting $\phat_n = p$ into \eqref{eq:general_adaptive} recovers the optimal CuSum test statistic \eqref{eq:cusum_stat}.

Based on the above, the central question for constructing a change detection test of the form \eqref{eq:general_adaptive} is designing predictors $\phat_n$ such that $\calR_\nu(\phat_n)$ is as small as possible for $n\geq\nu$. If $\{X_n\}$ was an i.i.d. sequence from $P$, the loss $\calR_1(\phat_n)$ or its cumulative version $\sum_{n=1}^N\calR_1(\phat_n)$ is known as KL-loss or (cumulative) predictive regret.
This framework has been extensively studied in sequential prediction,
predictive density estimation, and the minimum description length (MDL)
literatures \cite{Grunwald_MDL_tutorial,brown2008admissible}.
An interesting takeaway from these works is that the plug-in approach, while simple, may be suboptimal. Suppose that $\calP = \{P_\btheta : \btheta \in \bTheta\}$ is a parametric family of distributions, and $\btheta \sim \omega(\btheta)$ for some prior distribution $\omega$ and $\bX_{n-1}$ is an i.i.d. sample from $P_\btheta$. It is well known \cite{AITCHINSON_Goodness} that the Bayesian \emph{posterior predictive distribution} 
\begin{equation}\label{eq:bayes_post}
    \phat_n^\text{bayes}(X_n) = \int_\btheta p(X_n;\btheta)\omega(\btheta|\bX_{n-1})d\btheta
\end{equation}
minimizes $\Ex_{\btheta \sim \omega(\btheta)}[\calR_1(\phat_n)]$ over all predictive distributions $\phat_n$. In \eqref{eq:bayes_post}, $\omega(\btheta | \bX_{n-1})$ denotes the posterior distribution of $\btheta$ given data $\bX_{n-1}$. 

In practical problems, it may not make sense to talk about a distribution $\omega(\btheta)$ that generates $\btheta$. Regardless, there exist ``priors'' $\omega$ such that the resulting posterior predictive distribution \eqref{eq:bayes_post} dominates sensible plug-in predictors
\begin{equation}
    \phat^\text{plug-in}_n(X_n) = p(X_n;\bthetahat_n)
\end{equation} 
under purely frequentist criteria. As an example, if $P_\btheta$ is the standard univariate Gaussian distribution with unknown mean $\btheta \in \bbR$, it was shown that the maximum likelihood predictive plug-in $\phat^\text{plug-in}_n(X_n) = p(X_n;\bthetahat_n^\text{ML})$ is dominated by the posterior predictive distribution $\phat_n^\text{bayes}(X_n)$, computed by assuming an improper uniform prior $\omega(\btheta)\propto 1$ for $\btheta$. That is, for every $\btheta \in \bbR$
\begin{equation}
    \calR_1(\phat_n^\text{bayes}) < \calR_1(\phat^\text{plug-in}_n).
\end{equation}
Here, the plug-in predictive distribution corresponds to $\phat^\text{plug-in}_n(x) = \calN(x;\bthetahat^\text{ML}, 1)$, while the posterior predictive is given by $\phat_n^\text{bayes} = \calN(x; \bthetahat^\text{ML}, 1 + 1/(n-1))$. In particular, $\phat_n^\text{bayes}$ accounts for the uncertainty in the point estimate of the parameter, while this is ignored in the plug-in distribution. In general, constructing admissible or otherwise optimal predictive distributions is intricate even for Gaussian data models. It can be shown that all admissible rules correspond to Bayes posterior predictive distributions \cite{brown2008admissible}.

In the change detection context, the important quantity for effective detection is not $\calR_1(\phat_n)$, but $\calR_\nu(\phat_n)$ for $n \geq \nu$, where the change-point $\nu$ is unknown. That is, the procedure for generating the predictive distribution $\phat_n$ should be effective even when an unknown segment of the prior data $X_1,\dots,X_{n-1}$ is not drawn from $P$, but from the pre-change model $Q$. When computing the predictive density $\phat_n$, using all past data is most likely suboptimal, since it likely contains observations from $Q$. Determining how much past data should be used is itself nontrivial, because the optimal history length depends on the change-point $\nu$, which is precisely the unknown quantity to be detected.

\section{Predictive-Mixture CuSum (PM-CuSum) procedure}\label{sec:pred_cusum}
In this section we propose the Predictive-Mixture CuSum procedure that can effectively construct the predictive distribution $\phat_n$ in a change-point detection context. PM-CuSum maintains a set of candidate predictors $\{\phat^{(w)}_n\}_{w\in\calW}$ indexed by window lengths $w \in \calW$ for generating the predictive distribution $\phat_n$ for each time instance. The defining feature of each candidate predictor $\phat_n^{(w)}$ is that its predictive distribution is constructed based on the past $w$ observations, i.e. 

\begin{equation}
    \phat_n^{(w)}(x) = \phat(x|X_{n-1},\dots,X_{n-w}). 
\end{equation}
At time step $n$, the candidate predictive distributions are combined using weights $\{\pi_n^{(w)}\} _{w\in\calW} \in \calF_{n-1}$ independent of $X_n$ to form the combined predictive distribution
\begin{equation}\label{eq:pm_def}
    \phat_n(X_n) = \sum_{w\in\calW} \pi_{n}^{(w)}\phat^{(w)}_n(X_n), \qquad\qquad \sum_{w\in\calW}\pi_{n}^{(w)} = 1.
\end{equation}
This distribution is then used to compute the new likelihood ratio
\begin{equation}
    \lhat_n = \log \frac{\phat_n(X_n)}{q(X_n)},
\end{equation}
and update the PM-CuSum statistic with
\begin{equation}
    S_n = \max\{S_{n-1}, 0\} + \lhat_n, \quad S_{0} = 0.
\end{equation} 
The procedure declares a change when the statistic crosses a user-defined threshold $b > 0$:
\begin{equation}
    T_b = \inf\{n: S_n > b\}.
\end{equation}

The following lemma states that for any (predictable) choice of predictive densities $\{\phat^{(w)}_n\}$ and weights $\{\pi_n^{(w)}\}$, the average run length to false alarm (ARL) grows exponentially in $b$.
\begin{lemma}\label{thm:arl}
    For any $\{\phat^{(w)}_n\}$ and $\{\pi_n^{(w)}\}$, the ARL of PM-CuSum satisfies: 
    \begin{equation}
        \Ex_\infty T_b  \geq e^b.
    \end{equation}
    Therefore, setting $b = \log \gamma$ is sufficient to guarantee $\ARL(T_b) \geq \gamma$.
\end{lemma}
\begin{proof}
    The proof is given in Appendix~\ref{app:proof_of_arl}
\end{proof}

\begin{remark}

    At the start of the observation process for times $n < w$, an observation history of $w$ samples is not available. A simple workaround is to cap the window length at $n-1$ at time $n$. Alternatively, the final predictive distribution $\phat_n(X_n)$ in \eqref{eq:pm_def} can be formed by taking the sum over only $w < n$, i.e. $\{w \in \calW : w < n\}$. In this paper, we consider the first option, so that technically 
    \begin{equation}
        \phat_n^{(w)}(x) = \phat(x | X_{n-1},..,X_{\max\{1, n-w\}}).
    \end{equation}
    For brevity, this detail is suppressed from the notation in the rest of the paper.
\end{remark}

\subsection{Choosing the weights}
The PM-CuSum algorithm depends on two central design parameters: the candidate predictive distributions $\{\phat^{(w)}_n\}$ and the corresponding weights at each time. The choice of the predictive distributions is deferred to Section~\ref{sec:choose_predictives}. The proposed weighting strategies are general and applicable to any choice of predictive distributions.

At a general time $n = \nu + m$, the ideal window size to use for constructing the predictive distribution would be $m$, since windows longer than $m$ are ``contaminated'' by pre-change data. Of course, $m$ is not known, since $\nu$ is unobserved. Therefore, of the available options $\calW$, some windows $w\in\calW$ may be too long and biased by pre-change data, while others may be unnecessarily short. Moreover, the best choice varies over time: initially after the change, short windows are able to react to the change the quickest and provide a small positive drift for the test statistic, but after a sufficient number of samples from the post-change distribution are available, longer windows are to be preferred. An oracle weighting procedure would then place all its weight at time $n$ on whatever window provides the largest expected drift for the test statistic.

\begin{figure}[tb]
    \centering
    \includegraphics[width=0.7\linewidth]{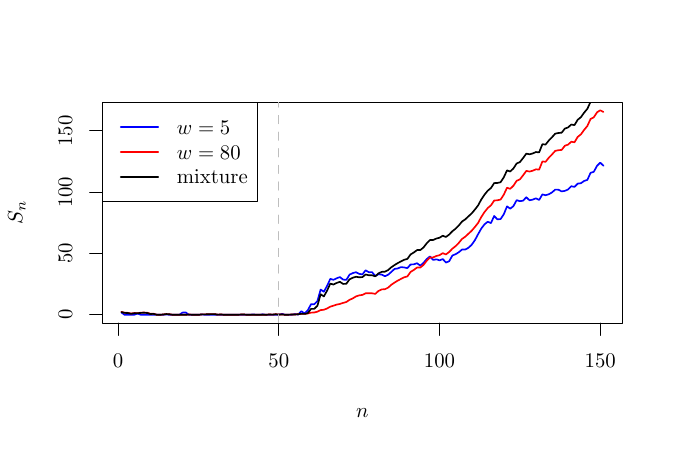}
    \caption{Illustration of the adaptive window weighting used in PM-CuSum. The red and blue lines correspond to adaptive change detection test statistics that utilize long and short windows, respectively. The statistic corresponding to the short window reacts quickly to the change but is eventually surpassed by the longer window. By adaptively combining the two statistics, a statistic larger than either of the individuals can be obtained.}\label{fig:cartoon}

\end{figure}
Figure~\ref{fig:cartoon} illustrates this tradeoff. The short-window statistic reacts quickly after the change at $\nu=50$, but its noisy estimates lead to slower long-run growth. The long-window statistic initially adapts more slowly because its window contains mostly pre-change observations, but it grows faster once the window is dominated by post-change data. The adaptively weighted statistic first tracks the short window and later shifts toward the long window, eventually exceeding both individual statistics.

The problem of choosing the best predictor at a given time may be viewed as an ``expert selection'' problem \cite{Herbester_fixed_share}. Expert-selection methods are typically used in online prediction and learning to track the best-performing strategy in a changing environment \cite{Cesa-Bianchi_Lugosi_2006}, but have not been widely adopted in the QCD literature. In the present setting, the experts are predictive distributions based on different window lengths. Each expert is scored by the predictive likelihood it assigns to the realized observation, $\phat^{(w)}_n(X_n)$. To account for the fact that the best predictor may change over time as discussed above, we adopt the Fixed Share (FS) algorithm \cite{Herbester_fixed_share} for calculating the weights. In the FS procedure, the weight $\pi_n^{(w)}$ of predictor $w$ at time $n$ is updated by
\begin{align}
    \tilde\pi_{n+1}^{(w)} &\propto \pi_n^{(w)}\phat^{(w)}_n(X_n) \\
    \pi_{n+1}^{(w)} &= (1-\alpha)\tilde\pi_{n+1}^{(w)} + \frac{\alpha}{|\calW|},\label{eq:fs_update}
\end{align}
with $\pi_0^{(w)}$ equal to some initial allocation of the weights, typically $1/|\calW|$ for all $w$. In the update, parameter $\alpha$ represents the rate at which we expect the best predictor to change. Intuitively, the algorithm retains a share $\alpha$ of the weight mass that is redistributed uniformly across predictors at each time step, ensuring continued adaptability. This prevents the weight of any predictor from vanishing. Pseudo-code for the proposed procedure is given in Algorithm~\ref{alg:pcusum_fs}.

A useful property of the Fixed Share algorithm is that it controls the ``switching regret'' \cite{adamskiy} relative to any sequence of chosen window lengths with a limited number of switches. That is, if $w_1,...,w_N$ is any sequence of windows in $\calW$ with at most $K-1$ switches, the FS algoritm 
\begin{equation}\label{eq:FS_bound}
    \sum_{n = 1}^N \log \frac{\phat_n^{(w_n)}(X_n)}{\phat_n(X_n)} \leq \log |\calW| + (K-1)\log(|\calW|-1) - (K-1)\log \alpha - (N-K)\log(1-\alpha).
\end{equation}
Equation~\eqref{eq:FS_bound} states that the likelihood assigned to the observations $X_1,...,X_N$ by the predictive distributions chosen using Fixed Share aggregation cannot be too far from the likelihood assigned using any sequence of windows (even the best in hindsight) with $K$ switches. This regret guarantee will play an important role in establishing the asymptotic detection delay properties of the proposed procedure. Relatedly,
recent work has shown that regret bounds are also a central tool in the
analysis of sequential hypothesis tests \cite{waudbysmith2025universallogoptimalitygeneralclasses}.

\begin{remark}
    It will be shown, that using FS to set the weights leads to strong theoretical properties and good empirical performance of the proposed change-detection procedures. The Fixed Share algorithm requires specifying the parameter $\alpha$, although it will be observed that the empirical performance of the algorithm does not seem to depend heavily on $\alpha$. Still, if one wants an adaptive algorithm without any tuning parameters, a simple choice is to maintain individual CuSum statistics $\{S^{(w)}_{n}\}$ for each predictor $w$, and choose the weights for time $n+1$ based on these statistics. For example, a ``follow-the-leader'' strategy would place all weight on ${\arg \max}_w \{S^{(w)}_{n}\}$. This approach would make the combined statistic behave approximately like the running maximum of the individual statistics. In contrast, Fixed Share aggregation can produce a statistic that exceeds each individual statistic by switching between predictors over time; see Figure \ref{fig:cartoon}.
\end{remark}

\begin{algorithm}[H]
\footnotesize
\caption{Predictive-Mixture CuSum}
\label{alg:pcusum_fs}
\begin{algorithmic}[1]
\Require Threshold $b>0$; set of windows $\mathcal W$;
share parameter $\alpha\in[0,1]$.
\State Initialize $S_0 \gets 0$;  initial weights $\pi^{(w)}_{1}\gets1/|\calW|$
\For{$n = 2,3,\ldots$}
    \ForAll{$w\in\mathcal W$}
        \State $w_n \gets \min\{w, n-1\}$
        \State Compute predictive likelihood $\phat^{(w)}_{n}(X_n) = \phat(X_n|X_{n-1},\dots,X_{n-w_n})$ 
    \EndFor
    \State Form aggregated predictive likelihood
        \[
            \phat_n(X_n) \gets \sum_{w\in\mathcal W}\pi^{(w)}_{n}\,\phat^{(w)}_{n}(X_n).
        \]
    \State Update the test statistic
            $$\lhat_n = \log\frac{\phat_n(X_n)}{q(X_n)}$$
        \[
            S_n \gets \max\{S_{n-1} , 0 \} + \lhat_n
        \]
    \If{$S_n > b$}
        \State \Return $T \gets n$ and declare change.
    \EndIf
    \State Fixed Share weight update:
    \ForAll{$w\in\mathcal W$}
        \State Unnormalized update $\tilde\pi^{(w)}_{n+1} \gets \pi^{(w)}_{n}\phat^{(w)}_{n}(X_n)$
    \EndFor
    \State Normalize $\tilde\pi^{(w)}_{n+1} \gets \tilde\pi^{(w)}_{n+1}\big/\sum_{w\in\mathcal W}\tilde\pi^{(w)}_{n+1}$ for all $w$
    \ForAll{$w\in\mathcal W$}
        \State Share step $\pi^{(w)}_{n+1} \gets (1-\alpha)\tilde\pi^{(w)}_{n+1} + \alpha/|\calW|$
    \EndFor
\EndFor
\end{algorithmic}
\end{algorithm}

\subsection{Choosing the predictive distributions}\label{sec:choose_predictives}

The performance of PM-CuSum depends on the predictive densities $\phat_{n}^{(w)}$ which determine the log-likelihood increments and thus the drift of the detection statistic. Thus, for each $\phat_n^{(w)}$, $w \in \calW$, the problem reduces to predictive density estimation based on $w$ samples $\bXw = (X_{n-1},\dots,X_{n-w})$ assumed to be drawn i.i.d. from $P$. As discussed in Sec.~\ref{subsec:cd_by_prediction}, the quality of the predictors is measured through KL-loss $\calR_1(\phat^{(w)}_n)$, which determines the post-change drift of the test-statistic $S_n$. While the proposed procedure can in principle be used in conjunction with any $\phat_{n}^{(w)}$, its KL-loss determines the asymptotic performance of the algorithm. In the following assumption, we require that in the stationary post-change regime, the KL-loss of the predictive distribution must decay polynomially in $w$.

\begin{assumption}\label{assumption:kl_loss} There exist constants $0 < \beta \leq 1$ and $C_1 < \infty$ such that the KL-loss of the predictor $\phat_n^{(w)}$ with window size $w$ satisfies, for $n > w$,
    \begin{equation}\label{assumption:kl_loss_eq}
        \calR_1(\phat^{(w)}_n) = \Ex_1^P\left(\log\frac{p(X_n)}{\phat_n^{(w)}(X_n)}\right) \leq \frac{C_1}{w^\beta}, \quad \text{for all } P \in \calP
    \end{equation}
\end{assumption}

Next, we give examples of predictors that satisfy \eqref{assumption:kl_loss_eq} and the associated $\beta$ in parametric and nonparametric cases. In the parametric case, the two important classes of predictors are Bayesian predictive distributions and the plug-in approach.
\begin{example}[Bayesian predictive densities.]
    Let $\mathcal{P}$ be a $k$-parameter exponential family, and $\omega$ a continuous prior with $\omega(\btheta)>0$ for all $\btheta\in\bTheta$. Then the predictor given by the posterior predictive distributions, defined by
    \begin{equation}\label{eq:bayes_pred}
        \phat_n^{(w)}(X_n) = \int_\bTheta p(X_n;\btheta)\omega(\btheta|\bXw)d\btheta
    \end{equation}
    satisfies Assumption \ref{assumption:kl_loss} with $\beta = 1$ under standard regularity conditions for finite-dimensional exponential families \cite{Grunwald_MDL_tutorial}.
\end{example}
\begin{example}[Plug-in maximum likelihood.]
    In the absence of conjugate priors or closed-form distributions, the full predictive density \eqref{eq:bayes_pred} may be computationally expensive to evaluate. In such cases, the plug-in approach provides a simpler but useful alternative. Let $\bthetahat_{n-1}^{(w)}$ denote an estimator of $\btheta$ based on $\bXw$. Then, the plug-in predictor is given by 
    \begin{equation}
        \phat^{(w)}_n(X_n) = p(X_n; \bthetahat_{n-1}^{(w)}).
    \end{equation}
    In $k$-parameter exponential families, $\phat_n^{(w)}$ satisfies Assumption~\ref{assumption:kl_loss} with $\beta=1$ if $\bthetahat_{n-1}$ is chosen as a (smoothed) maximum likelihood estimator \cite{Grunwald_MDL_tutorial}, for example.  
\end{example}

When $\calP$ is a nonparametric class of distributions, many ways of estimating $P \in \calP$ exist. Kernel density estimators (KDEs) are among the most common such methods. A properly tuned KDE satisfies \eqref{assumption:kl_loss_eq} with a $\beta$ that depends on the dimension and smoothness of the distributions in $\calP$.
\begin{example}[Kernel density estimators (KDE)]
Let $\calP_r$ be the $r$-Hölder density class of probability distributions (see \cite{Liang2024QuickestEstimation}\cite[pp. 5]{tsybakov2008nonparametric} for the exact definition). Intuitively, $\calP_r$ contains distributions with densities that are $\lfloor r \rfloor$ times differentiable and the $\lfloor r \rfloor$-order derivative has bounded local variation.

The KDE is defined as
\begin{equation}
    \phat_n^{(w)}(X_n) = \frac{1}{w\prod_{i=1}^k h^{(i)}}\sum_{j=n-w}^{n-1}\prod_{i=1}^k K\left(\frac{X_n^{(i)}-X_j^{(i)}}{h^{(i)}}\right),
\end{equation}
where $K(\cdot)\geq 0$ is a \emph{kernel function} and $h^{(i)}, 1\leq i \leq k$ are smoothing parameters. If $K$ and $h$ are properly chosen, it is shown in \cite[Lemma 1]{Liang2024QuickestEstimation} that the KDE predictive satisfies \eqref{assumption:kl_loss_eq} such that the KL-loss decreases with exponent 
\begin{equation}
    \beta = \frac{2r}{2r + k} < 1,
\end{equation}
for all $P \in \calP_r$.
\end{example}

Besides the assumption on KL-loss, a technical assumption on the tail-behavior of the likelihood ratio increment $\lhat_n$ is made for the purposes of theoretical analysis.
\begin{assumption}\label{assumption:overshoot_bound}
    It is assumed that there exists a constant $C_2$ such that for all $n$
    \begin{equation}\label{eq:overshoot_eq}
        \sup_{t\geq 0}\Ex_\nu\left( \lhat_n -t  \, \Big| \,   \lhat_n\geq t, \calF_{n-1}  \right) \le C_2 \quad \text{almost surely.}
    \end{equation}
\end{assumption}
Assumption~\ref{assumption:overshoot_bound} is used solely for bounding the overshoot of the test statistic over the stopping threshold. In some contexts, the expectation $\Ex(Z -t | Z\geq t)$ is known as the mean residual life (MRL) of a random variable $Z$. The assumption requires that the MRL of $\lhat_n$ is bounded for all $t$ and is satisfied, for example, if its distribution is log-concave \cite{Barlow}. A similar assumption was made by Wald in his foundational work on sequential analysis \cite{Wald1947SequentialAnalysis}, and recently in e.g. \cite{warner2024worst}. However, we believe that Assumption~\ref{assumption:overshoot_bound} could be weakened without affecting the asymptotic detection delay results of Theorem~\ref{thm:delay_thm} below, see discussion in Appendix~\ref{appendix:overshoot_discussion}.

To obtain the smallest asymptotic growth rate of detection delay, it suffices to consider a logarithmically spaced collection of window lengths. Accordingly, we let
\begin{equation}\label{eq:W_def}
    \calW = \{2^r : r = 1,2,\ldots,\lceil \log_ 2 b\rceil\}
\end{equation}
so that the largest window is of order $b=O(\log \gamma)$ and the total number of windows $|\calW| = O(\log b) = O(\log \log \gamma)$ grows slowly w.r.t $\gamma$.

\begin{theorem}\label{thm:delay_thm}
    If the predictors $\phat_n^{(w)}$ satisfy Assumptions~\ref{assumption:kl_loss} and \ref{assumption:overshoot_bound}, the set of windows is chosen as in \eqref{eq:W_def}, and $\alpha = 1/b$, then detection delay of PM-CuSum as $\gamma \to \infty$ is given by
    \begin{equation}\label{eq:thm_beta1}
        \delay{T_b} = \frac{\log \gamma}{\KL{p}{q}} +O((\log\log \gamma)^2),   \quad \text{if } \beta = 1
    \end{equation} 
    \begin{equation}
        \delay {T_b} = \frac{\log \gamma}{\KL{p}{q}} +O((\log \gamma)^{1-\beta}), \quad \text{if } 0 < \beta < 1 
    \end{equation} 
    for all $P \in \calP$. Therefore, the procedure is first-order asymptotically optimal as $\gamma \to \infty$ when $\beta > 0$.

\end{theorem}
\begin{proof}
    The proof is given in Appendix~\ref{app:proof_of_delay}
\end{proof}

In the change-detection literature, first-order optimality is usually given in the form $\delay T = \log\gamma/\KL{p}{q}(1+o(1))$. Theorem~\ref{thm:delay_thm} states that for the proposed algorithm, the $o(1)$ term goes to zero at a particularly fast rate. In particular, the delay expressions can be rewritten as
\begin{equation}\label{eq:add_remainder_1}
    \delay{T_b} = \frac{\log\gamma}{\KL{p}{q}}\left(1 + O \left(\frac{(\log\log\gamma)^2}{\log\gamma}\right)\right) \quad \text{if } \beta = 1
\end{equation}
\begin{equation}
    \delay{T_b} = \frac{\log\gamma}{\KL{p}{q}}\left(1 + O \left((\log \gamma)^{-\beta}\right)\right) \quad \text{if } 0 < \beta < 1.
\end{equation}
When $\beta = 1$, the WL-CuSum procedure \cite{XIE_2023} achieves a remainder term of $O((\log\gamma)^{-1/2})$ in \eqref{eq:add_remainder_1}, larger than what is achieved by the proposed method. When $\beta < 1$, the remainder term with a single window has been shown to be of order $O((\log\gamma)^{-\rho})$, where $\rho = \max_{\kappa \in [0,1]}\min(\kappa\beta, 1-\beta) < \beta$ \cite{Liang2024QuickestEstimation}, which is again larger than the remainder term achieved by the proposed method. Thus, adaptively combining predictors of different window lengths leads to better asymptotic performance than what can be achieved with any single fixed window size. Moreover, the number of windows required by the proposed algorithm is only $O(\log\log\gamma)$.

In $k$-parameter exponential families, the minimax optimal second-order term in \eqref{eq:thm_beta1} is known to equal $C\log\log\gamma$ \cite{TARTAKOVSKY_BOOK}, and can be achieved (including the optimal constant) by some algorithms, such as \cite{LORDEN_2005} and \cite{CAO_2018}. However, these algorithms are computationally expensive, essentially starting a new adaptive test at each time instance such that the computational complexity grows linearly in time. For single-parameter exponential families, the restart-based estimation procedure of \cite{LORDEN_2008} achieves the optimal second-order term in terms of Pollak's delay. While it may be possible to extend the result to e.g. multi-parameter exponential families, giving an upper bound with respect to Lorden's worst-worst-case delay, as done in this paper, is likely more challenging. This is because it is hard to rule out the possibility of worst-case pre-change data causing the estimator to take unfavorable values at the time the change happens.

\subsection{Setting the Fixed Share parameter}

The proposed algorithm involves setting a parameter $\alpha$ that determines how much of the total weight mass is redistributed uniformly among the windows at each time. In an informal sense, $\alpha$ defines the length of the ``memory'' of the weight assignment procedure. In the theoretical analysis, we considered a fixed value of $\alpha$ (that vanishes as $\gamma \to \infty$), but for practical use it may be useful to let the $\alpha$ in \eqref{eq:fs_update} depend on $n$. 

Intuitively, when the change has not occurred, the algorithm should be able to quickly adapt to any window size being effective. Before the change, no window length is systematically better for detection (since there is nothing to detect), so premature concentration of the weights is undesirable. Conversely, when the change has happened, the procedure should eventually ``zoom in'' on the window lengths $w$ where $\phat_n^{(w)}$ has recently been the largest. A natural measure of evidence of the change having happened is the current value of the test statistic $S_n$. One simple way of setting $\alpha_{n+1}$ given $S_n$ is to use a logistic function transform
\begin{equation}\label{eq:alpha_eq}
    \alpha_{n+1} = 1- \frac{1}{1+e^{-S_n^+}} = \frac{1}{1+e^{S_n^+}} \in (0, 1/2],
\end{equation}
where $S_n^+ = \max\{0, S_n\}$.
When $S_n$ is small, $\alpha_{n+1}$ is large and the fixed share update can quickly adjust the weights between windows based on new data. When $S_n$ is larger, indicating that a change may have occurred, $\alpha_{n+1}$ is reduced, meaning that less of the weights are redistributed, and more is placed on windows that have recently made good predictions (that caused $S_n$ to grow in the first place). 

\begin{figure}[htpb]
    \centering
    \includegraphics[width=0.6\linewidth]{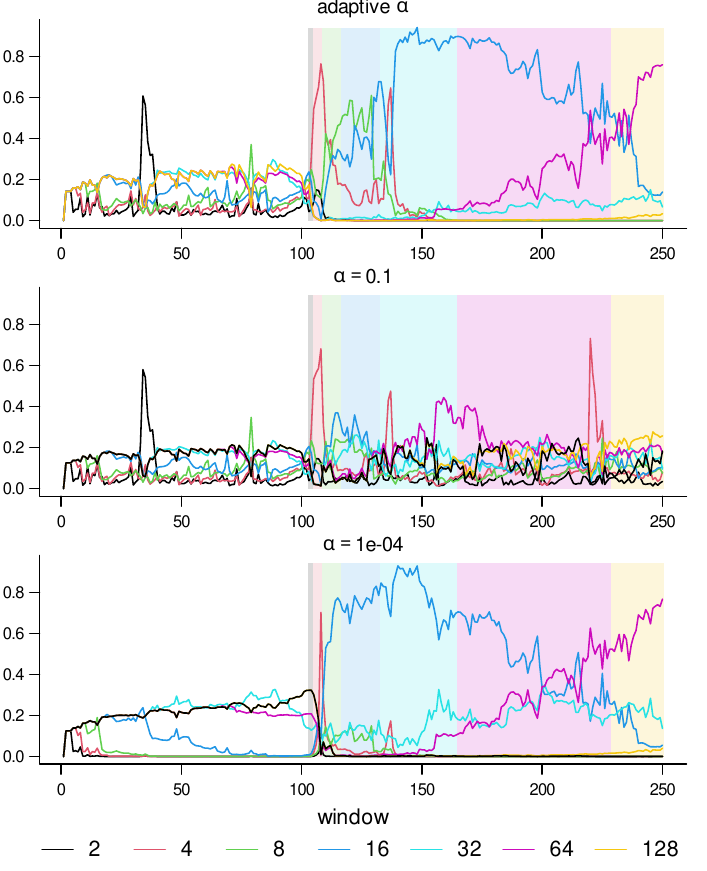}
    \caption{The weights assigned to different windows by the PM-CuSum procedure with different specifications of $\alpha$ as a function of time on a sample run. The background color indicates the (unknown) most effective window length at a given time. An adaptive choice of $\alpha$ keeps the weights roughly equally distributed before the change at time $n=100$, and places weights on the currently effective window sizes after the change.}\label{fig:alpha_fig}
\end{figure}

The effect of the choice $\alpha$ is illustrated in Figure~\ref{fig:alpha_fig}. Three versions of the PM-CuSum algorithm with different $\alpha$ are applied to the same set of data where a change happens at time $n = 100$. In the plots, the lines represent the weights $\pi_n^{(w)}$ assigned to each window $w \in \calW$ as a function of time. In the top plot, $\alpha$ is set adaptively using $\eqref{eq:alpha_eq}$, in the middle plot $\alpha$ is set to a large fixed value of $\alpha = 0.1$, and in the bottom plot $\alpha = 10^{-4}$ is small. The set of windows is $\calW = \{2,4,8,16,32,64,128\}$. The other simulation parameters are as in Sec.~\ref{subsec:parallel_compare}, though the details are inconsequential for the purposes of this illustration. The colored background from time $n=100$ onward represents the ``optimal'' window at any given time; the largest $w \in \calW$ such that $n-100 \leq w$. For instance, between $n = 132$ and $n = 164$ the ``optimal'' choice of the options in $\calW$ would be $w = 32$ (indicated by the light-blue color in the plots), as it is the largest window that is not biased by pre-change data\footnote{The term ``optimal'' is used informally here, since the best window in terms of $\mathcal{R}_\nu(\phat_n^{(w)})$ would depend on the severity of the bias caused by the pre-change data on the predictive distribution. For example, at time $n = \nu + 63$, it is likely that $\phat_n^{(w)}$ with $w=64$ generates better predictions than $w=32$, since the bias from the single pre-change observation is likely negligible.}. 

In the bottom plot of Figure~\ref{fig:alpha_fig}, where $\alpha$ is small, the weights start to concentrate on the larger window lengths in the pre-change regime. This is because the long windows make good predictions of new data in a stationary setting. Of course, having most of the weight concentrated on the long windows when the change happens is suboptimal. In the middle plot, where $\alpha$ is large, the weights remain relatively equally distributed in the pre-change regime. Unfortunately, the large $\alpha$ also means that the weights may not be effectively placed on the best values of $w \in \calW$ in the post-change regime. The time-varying choice of $\alpha$, displayed in the top plot, causes the weights to behave as desired in both the pre- and post-change regimes. Before the change, the weights remain dispersed, but concentrate on the effective values of $w$ after the change.

\section{Performance evaluation with simulation}\label{sec:simulations}

\subsection{Comparison to the parallel WL-CuSum}\label{subsec:parallel_compare}

We first isolate the effect of adaptive window aggregation by comparing PM-CuSum with the parallel WL-CuSum procedure of \cite{XIE_2023}. Since the ``optimal'' window size of WL-CuSum is generally unknown in practice, the parallel WL-CuSum procedure works by running multiple WL-CuSum procedures in parallel, each with a different window size, and stopping when any of the individual procedures stops. This approach was empirically shown to be effective in \cite{XIE_2023}. 

We consider a scenario where pre-change distribution $q = \calN(\bm 0, I_k)$, post-change distribution $p = \calN(\btheta, I_k)$ with $\norm{\btheta} = 1$, for dimension $k \in \{5, 100\}$. The case $k=100$ is a more difficult detection problem than $k = 5$, as the SNR is lower. For a fair comparison, both the PM-CuSum and parallel WL-CuSum utilize the MLE plug-in distribution, given by
\begin{align}\label{eq:gaus_mle_plug}
    \phat_n^{(w)}(X_n) &= \calN(X_n; \bthetahat^{(w)}_n, I_k) \\
    \bthetahat^{(w)}_n &= \frac{1}{w}\sum_{j=n-w}^{n-1}X_j.
\end{align}

\begin{figure}[htpb]
    \centering
    \begin{subfigure}[b]{0.49\linewidth}
        \centering
        \includegraphics[width=\linewidth]{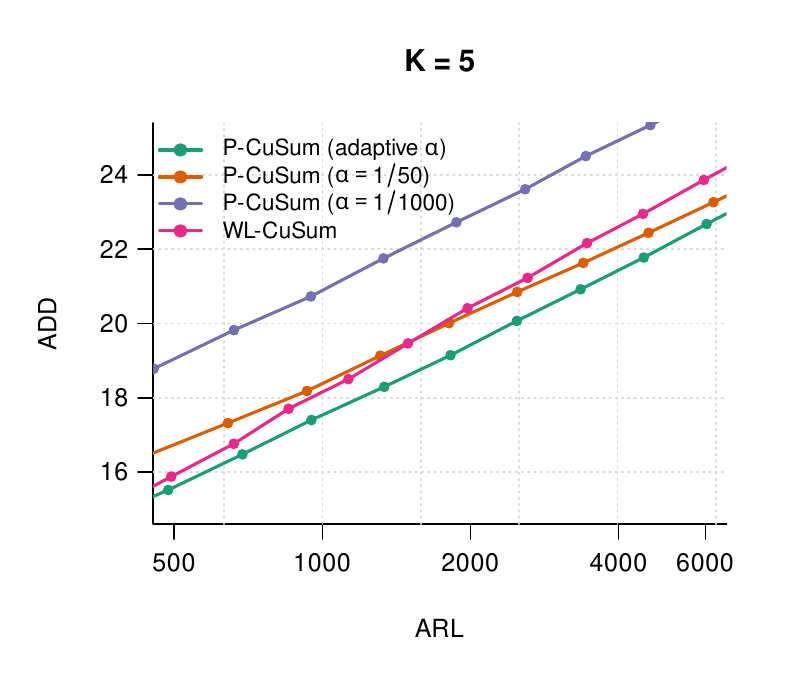}
        \caption{k = 5}
        \label{fig:predictive_vs_parallel_1}
    \end{subfigure}
    \hfill
    \begin{subfigure}[b]{0.49\linewidth}
        \centering
        \includegraphics[width=\linewidth]{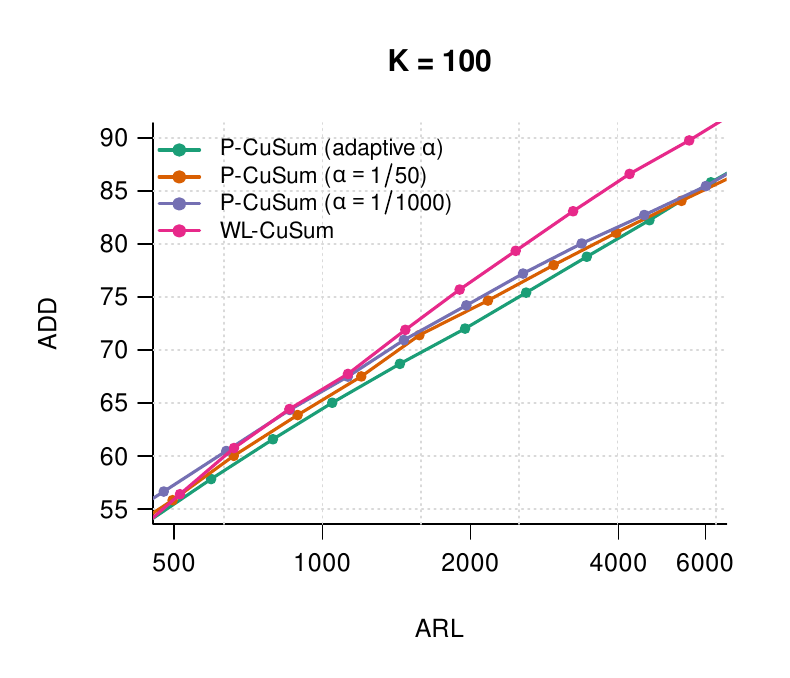}
        \caption{k = 100}
        \label{fig:predictive_vs_parallel_2}
    \end{subfigure}
    \caption{Average detection delay versus ARL for PM-CuSum and parallel WL-CuSum for different number of data streams $k$. The average detection delay of PM-CuSum with a time-varying choice of $\alpha$ is the smallest for all ARL levels for both $k$.}
    \label{fig:predictive_vs_parallel}
\end{figure}

For both algorithms, the set of window sizes is set as $\calW = \{2^r : 1 \leq r \leq 7\}$. Simulations were also run with other sets of windows, and the relative performance of the algorithms was similar. We consider three versions of the  PM-CuSum procedure, one with weight-reallocation parameter $\alpha = 1/50$, one with $\alpha = 1/1000$, and one where the parameter $\alpha$ is set adaptively as in \eqref{eq:alpha_eq}. The change-point is set to $\nu = 100$. The observed empirical detection delays of the procedures for a range of ARL levels are displayed in Figure~\ref{fig:predictive_vs_parallel}. The detection delay of PM-CuSum with a time-varying $\alpha$ is the smallest across all levels of ARL and for both values of $k$. A fixed $\alpha$-parameter that is too large can lead to large detection delays compared to other algorithms when the change is relatively easy to detect, as seen from Fig.~\ref{fig:predictive_vs_parallel_1}. When the detection delay is longer in Fig.~\ref{fig:predictive_vs_parallel_2}, all variants of the PM-CuSum have smaller detection delay than WL-CuSum with the same set of windows.

\subsection{Sparsity-adaptive detection of Gaussian mean shifts}\label{sec:gaus_shift}

In this section, we examine the effect of using full predictive distributions in place of plug-in method \eqref{eq:gaus_mle_plug}, and evaluate the performance against state-of-the-art algorithms. Before the change, the observations $X_n\in \bbR^k$ are assumed to follow $\calN(\bm 0, \sigma^2I_k)$, and $\calN(\btheta, \sigma^2I_k)$ after the change for some unknown mean vector $\btheta \in \bbR^k$. This problem is well-studied in the literature, and especially the setting where $\btheta$ is a sparse vector, i.e. only an unknown subset of the streams experiences the change, has attracted attention. In \cite{XIE_2013} and \cite{chan2017optimal}, the GLR-CuSum test is modified to improve its performance when $\btheta$ contains only a small fraction of non-zero components. Conversely, in the case of dense $\btheta$, it has been observed in \cite{halme2025quickest, WANG_2015}, that shrinkage estimation can be used to effectively reduce detection delay. Recently, \cite{chen2022high} proposed an algorithm that works well for both sparse and dense $\btheta$.

For PM-CuSum to adapt to both sparse and dense changes, we use two predictor families $\{\phat_n^{(\text{s},w)}\}_{w \in \calW}$ and $\{\phat_n^{(\text{d},w)}\}_{w \in \calW}$ (short for ``sparse'' and ``dense'', respectively).
\subsubsection*{Predictor for dense changes}

The predictor for dense changes is formed by setting an independent Gaussian prior distribution 
\begin{equation}\label{eq:gaus_prior}
    \omega(\theta^{(j)}) = \calN(\btheta^{(j)};\mu_0,\tau_0^2), \quad \text{for all } j = 1,...,k
\end{equation}
on the components of $\btheta$. Based on a sample $\bXw$, and denoting the within-window sample mean of the $j$th component by $\bar X_w^{(j)} = (1/w)\sum_{i=n-w}^{n-1}X_i^{(j)}$, the dense predictive distribution is then
\begin{align}
    \phat_n^{(\text{d},w)}(X_n) &= \int p(X_n;\btheta)\omega(\btheta|\bXw)d\btheta \\
    &=\prod_{j=1}^k \calN(X_n^{(j)}; \mu_n^{(j)}, \sigma_n^2 + \sigma^2),
\end{align}
where 
\begin{align}
    \sigma^2_n &= \left(\frac{w}{\sigma^2} + \frac{1}{\tau_0^2}\right)^{-1}\label{eq:gaus_pred_var} \\
    \mu_n^{(j)} &= \sigma_n^2\left(\frac{\mu_0}{\tau_0^2} + \frac{w\bar X^{(j)}}{\sigma^2}\right).\label{eq:gaus_pred_mean}
\end{align}
The hyperparameters $\mu_0$ and $\tau^2_0$ in \eqref{eq:gaus_prior} can be estimated in an empirical Bayes fashion using maximum marginal likelihood, so that
\begin{align}
    (\hat\mu_0, \hat\tau^2) = \underset{{\mu_0, \tau^2_0}}{\arg \max} \int p(\bXw; \btheta)\omega(\btheta|\mu_0, \tau^2)d\btheta.
\end{align}
Routine computations show that the maximizers are
\begin{align}
    \hat\mu_0 &= \frac{1}{k}\sum_{j=1}^k\bar X^{(j)} \\
    \hat\tau^2_0 &= \max\left\{0, \frac1k\sum_{j=1}^k(\bar X^{(j)}-\hat\mu_0)^2 - \frac{\sigma^2}{w}\right\},
\end{align}
which can be inserted into \eqref{eq:gaus_pred_var}--\eqref{eq:gaus_pred_mean}.

\subsubsection*{Predictor for sparse changes}

For prediction when $\btheta$ may be sparse, we use a predictive distribution from coordinatewise spike-and-slab model. Predictive distributions for Gaussian data constructed from spike-and-slab priors were recently studied in \cite{rockova2024adaptivebayesianpredictiveinference} and shown to have strong theoretical properties in terms of KL-loss.  The results of \cite{rockova2024adaptivebayesianpredictiveinference} are stated in terms of Bayesian predictive densities and their KL-risk, but closed-form expressions for the predictive likelihoods with a Laplace-slab used below are not provided. We therefore derive these expressions explicitly, yielding a computationally tractable sparse predictive density that can be evaluated online and combined with the dense predictor within PM-CuSum.

Let \(A^{(j)} \in \{0,1\}\) indicate whether coordinate \(j\) is affected by the change. We assume
\begin{equation}
A^{(j)} \sim \operatorname{Bernoulli}(\eta),
\end{equation}
independently over \(j=1,\ldots,k\). Conditional on \(A^{(j)}\), the post-change mean component satisfies
\begin{equation}
\theta^{(j)} =
\begin{cases}
0, & A^{(j)}=0,\\
\vartheta^{(j)}, & A^{(j)}=1,
\end{cases}
\qquad
\vartheta^{(j)} \sim g_\lambda .
\end{equation}
Here \(\eta\in[0,1]\) is the prior probability that a coordinate is affected, and \(g_\lambda\) is a slab density controlling the magnitude of the nonzero components. In the literature on sparse Gaussian estimation, it has been demonstrated that the slab-distribution $g_\lambda$ needs to have heavy enough tails to effectively identify the few but possibly large non-zero components of $\btheta$ \cite{Castillo_haystack}. We use a Laplace slab
\begin{equation}\label{eq:lapl_dist}
g_\lambda(\vartheta)=\frac{\lambda}{2}\exp(-\lambda|\vartheta|).
\end{equation}
For a fixed window length \(w\), define
\begin{equation}
\bar X^{(j)}
=
\frac{1}{w}\sum_{i=n-w}^{n-1}X_i^{(j)}
\end{equation}
as the sample mean within the window. Throughout this subsection, the dependence of various quantities on the window length $w$ is suppressed for notational simplicity.

The marginal densities of $\bar X^{(j)}$ when $A^{(j)} = 0$ and $A^{(j)}=1$ are given, respectively, by
\begin{equation}
m_{0}(\bar X^{(j)}) =p(\bar X^{(j)} |A^{(j)}=0) = \calN(\bar X^{(j)};0,\sigma^2/w),
\end{equation}
and
\begin{equation}\label{eq:m_dense}
m_{1}(\bar X^{(j)}) =p(\bar X^{(j)} |A^{(j)}=1)=\int \calN(\bar X^{(j)};\vartheta,\sigma^2/w)g_\lambda(\vartheta)\,d\vartheta.
\end{equation}
Then the posterior probability that coordinate \(j\) is affected is
\begin{equation}\label{eq:post_prob}
\rho(z)=\mathbb P(A^{(j)}=1\mid \bar X^{(j)}=z)
=
\frac{\eta m_{1}(z)}
{(1-\eta)m_{0}(z)+\eta m_{1}(z)}.
\end{equation}
Conditional on \(A^{(j)}=1\), the posterior density of the mean $\theta^{(j)}$ is
\begin{equation}
p(\theta^{(j)} = \vartheta\mid \bar X^{(j)} = z, A^{(j)} = 1)
=
\frac{\calN(z;\vartheta,\sigma^2/w)g_\lambda(\vartheta)}
{m_{1}(z)}.
\end{equation}
The corresponding posterior predictive density for the new observation $X_n^{(j)} = x$ is then
\begin{equation}\label{eq:u_def}
u(x\mid z) := p(X_n^{(j)}=x|\bar X^{(j)}=z, A^{(j)}=1)= \int \calN(x;\vartheta,\sigma^2)\frac{\calN(z;\vartheta,\sigma^2/w)g_\lambda(\vartheta)}
{m_{1}(z)}d\vartheta
\end{equation}
Therefore, the coordinatewise sparse predictive density for $X^{(j)}_n = x$ given $\bar X^{(j)} = z$ is
\begin{equation}
\hat p_{s,w}(x\mid z)
=
(1-\rho(z))\calN(x;0,\sigma^2)+\rho(z)u(x\mid z).
\end{equation}
Closed-form expressions for $u(x\mid z)$ and $\rho(z)$ are derived in Appendix~\ref{appendix:gauss_shift_predictives}.
Since the coordinates are conditionally independent, the full sparse predictive density is
\begin{equation}
\hat p_n^{(s,w)}(X_n)
=
\prod_{j=1}^k
\hat p_{s,w}\!\left(X_n^{(j)}\mid \bar X^{(j)}\right).
\end{equation}
The hyperparameter $\eta$ can again be set in an empirical Bayes fashion by maximizing the marginal likelihood of $\bXw$, that is,
\begin{equation}
    \hat\eta = \underset{\eta \in [0,1]}{\arg \max} \prod_{j=1}^k\left[(1-\eta)m_0(\bar X^{(j)}) + \eta \, m_1(\bar X^{(j)})\right].
\end{equation} 
In Appendix~\ref{appendix:gauss_shift_predictives}, an efficient method for computing $\hat{\eta}$ is presented. The Laplace hyperparameter $\lambda$ could also be set using maximum marginal likelihood. However, in the simulations we set a fixed value $\lambda = 0.5$.

\subsubsection*{Combining the predictors}
Combining the sparse and dense predictors to form the final predictive $\phat_n(X_n)$ used in the PM-CuSum test is immediate within the proposed weighting approach. If $\calW$ denotes the set of windows used with both predictors, we require a total of $2|\calW|$ weights $\{\{\pi_n^{(\text{d}, w)}\}_{w\in\calW},\{\pi_n^{(\text{s}, w)}\}_{w\in\calW}\}$ such that 
\begin{equation}
    \sum_{w \in \calW} \pi_n^{(\text{d}, w)} + \sum_{w \in \calW} \pi_n^{(\text{s}, w)} = 1
\end{equation}
and
\begin{equation}
    \phat_n(X_n) =\sum_{w\in\calW}\pi_n^{(\text{d}, w)}\phat_n^{(\text{d},w)}(X_n) + \sum_{w\in\calW}\pi_n^{(\text{s}, w)}\phat_n^{(\text{s},w)}(X_n).
\end{equation}

\subsubsection{Simulation results}

The performance of PM-CuSum is compared with other algorithms proposed in the literature for high-dimensional Gaussian mean change. To evaluate the performance of the methods in both sparse and dense settings, we set the true parameter $\btheta \in \bbR^k$ using the following process on each Monte Carlo run. For a given number of affected sensors $s \in \{1,\dots,k\}$, we first select the set $\mathcal{S}$ of $s$ affected components randomly. Then we sample $\btheta = \bZ/\norm{\bZ}_2$, where $\bZ = (Z^{(1)},\dots,Z^{(k)})$ has independent components satisfying $Z^{(j)} \sim \calN(0,1)_{\ind{j\in \mathcal{S}}}$. Hence, $\norm{\btheta}_2 =1$ by construction. A similar setup was considered in \cite{chen2022high}. We set $k = 100, \sigma^2=1$ and consider values of $s \in \{1,5,10,20,50,100\}$.

The proposed PM-CuSum algorithm is run using the two predictors as described in Section~\ref{sec:gaus_shift}. Additionally, we implement a ``naive'' version of the algorithm that utilizes only the plug-in MLE predictive distribution \eqref{eq:gaus_mle_plug}. The set of window sizes is chosen as $\calW = \{2^j : 1\leq j\leq 7\}$. The procedure is compared with three methods designed specifically for Gaussian mean change detection, namely those proposed by Xie and Siegmund (XS) \cite{XIE_2013}, Chan \cite{chan2017optimal} and the OCD procedure \cite{chen2022high}. For these algorithms, we utilize their implementations in the \textbf{ocd} R package \cite{ocd_package}. Additionally, we benchmark against the windowed GLR-CuSum \cite{LAI_1998} procedure. The XS and Chan algorithms require a tuning parameter $p_0$ that represents the (hypothesized) proportion of streams that undergo a change. This parameter is set to a default value $1/\sqrt k = 0.1$, as suggested by the authors. The window length for XS, Chan and GLR-CuSum is set to $w=200$, which is much larger than the detection delay encountered in the considered simulations. In these three procedures, windowing is used only to reduce computational complexity, and the choice does not affect statistical performance as long as the window is large enough. For all procedures, the detection threshold is calibrated via 2000 Monte Carlo simulations such that ARL-level $\gamma = 5000$.

Figure~\ref{fig:compare_sparsity} presents the average detection delays for all procedures for a varying number of affected sensors $s$. In general, all procedures except for GLR-CuSum have a smaller detection delay when the change is sparse, even if $\norm{\btheta}_2$ (and hence KL-divergence between pre- and post-change) is the same for all values of $s$. This is because the algorithms are designed to utilize this extra structure present in $\btheta$. Across all levels of sparsity, the proposed PM-CuSum performs competitively with the alternatives. In the sparse regime, the observed detection delay is close to that of Chan and XS, which are specifically designed for sparse changes. In the dense case of $s \in \{50,100\}$, PM-CuSum performs better than sparsity-aware alternatives. Notably, the version of PM-CuSum utilizing the plug-in MLE predictive distribution performs much worse than the version utilizing the full predictive distribution across all levels of sparsity.

Moreover, the flexible nature of predictive distributions can take advantage of additional structure in $\btheta$. In Figure~\ref{fig:compare_sparsity_delta} we repeat the setup of Figure~\ref{fig:compare_sparsity}, except when sampling $\btheta$ we select $Z^{(j)} \sim \mathcal{N}(\delta, 1)_{1\{j\in \mathcal{S} \}}$ with $\delta = 0.3$. With $\delta > 0$, the non-zero components of $\btheta$ are slightly biased toward positive values, but still $\norm{\btheta}_2 = 1$. The results are presented in Figure~\ref{fig:compare_sparsity_delta}. For sparse changes the results are similar, but in the dense case of $s = {100}$, the proposed PM-CuSum algorithm performs noticeably better than alternatives except GLR-CuSum. Interestingly, the delay curve as a function of $s$ is non-monotone for PM-CuSum, but either increasing or approximately constant for the other methods. This follows from the use of two predictor families: the sparse predictor is well matched to changes affecting few components, and the dense predictor is well matched to changes affecting many components, especially when the values of the affected components are similar. For intermediate sparsity levels, neither model is as well matched to the post-change distribution, resulting in larger delay.
\begin{figure}[htpb]
    \centering
    
    \begin{subfigure}[t]{0.48\linewidth}
        \centering
        \includegraphics[width=\linewidth]{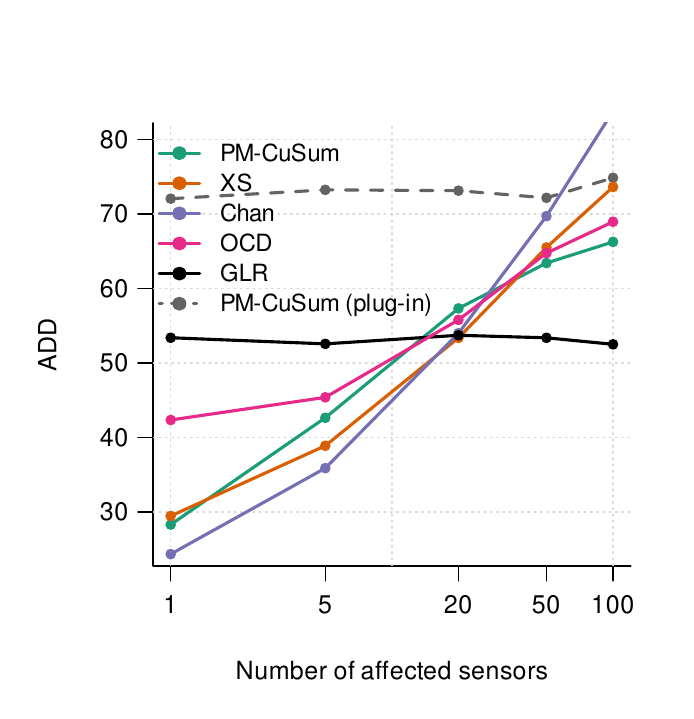}
        \caption{$\delta = 0$}
        \label{fig:compare_sparsity}
    \end{subfigure}
    \hfill
    \begin{subfigure}[t]{0.48\linewidth}
        \centering
        \includegraphics[width=\linewidth]{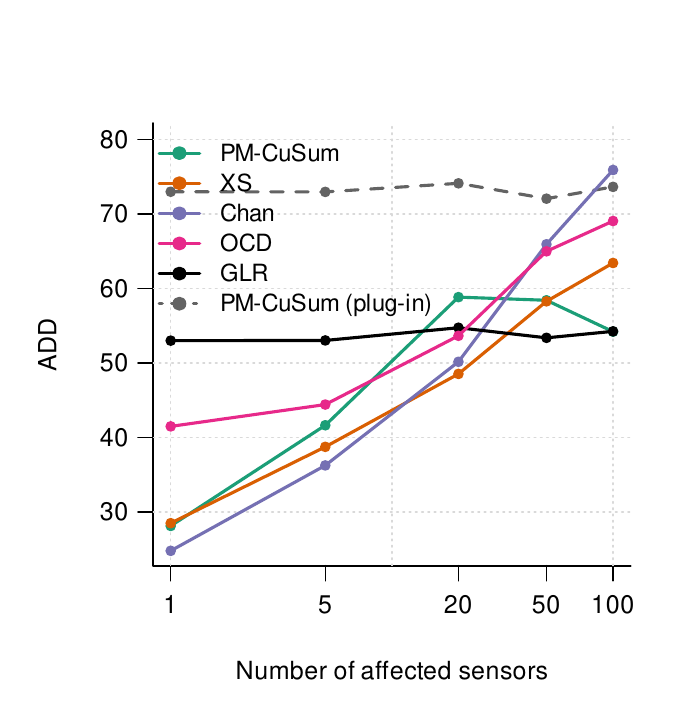}
        \caption{$\delta = 0.3$}
        \label{fig:compare_sparsity_delta}
    \end{subfigure}
    
    \caption{Comparison of detection delays of procedures across different levels of sparsity. The proposed PM-CuSum detector performs competitively across all levels of sparsity.}
\end{figure}

\section{Conclusion}\label{sec:conclusion}

In this paper, we have proposed PM-CuSum, a novel procedure for sequential change detection when the post-change distribution is unknown. By adaptively combining predictive distributions computed from different lengths of past data, the procedure can effectively detect changes while remaining computationally efficient. Analytically, it was established that the proposed procedure is first-order asymptotically optimal with respect to Lorden's delay under general conditions. Additionally, it was shown that the second-order term of the asymptotic detection delay is smaller than what can be achieved by any fixed single window. The results apply also to the case where the family of post-change distributions is nonparametric, as long as the distribution can be estimated with Kullback-Leibler loss that decreases sufficiently fast as a function of the sample size. The analysis focused on the i.i.d. setting, but for practical use, the proposed procedure can easily be extended to the case of dependent data by using predictors that take into account the dependence structure.

Additionally, we have made the observation that in the change-point detection problem with an unknown post-change distribution, the relevant notion of estimation error is the KL-divergence between the true and estimated post-change distribution. In this regard, it is known that simple plug-in distributions can be inferior to predictive distributions that take into account the uncertainty in the estimation of the post-change distribution. We demonstrated this phenomenon in a simulation example of detecting a change in the mean of high-dimensional Gaussian data, where the proposed PM-CuSum procedure with predictive distributions considerably outperforms the plug-in method, and performs competitively with state-of-the-art algorithms designed specifically for this problem.

While we have observed both empirically and theoretically that the Fixed Share algorithm works well for combining the predictors, in the future it would be of interest to study the applicability of other expert selection algorithms, such as those presented in \cite{jun_2017}. Finally, an interesting line of future work is extending ideas of this work to the case of a composite pre-change distribution.

\bibliographystyle{IEEEtran}
\bibliography{refs}

\begin{thebibliography}{10}
\providecommand{\url}[1]{#1}
\csname url@samestyle\endcsname
\providecommand{\newblock}{\relax}
\providecommand{\bibinfo}[2]{#2}
\providecommand{\BIBentrySTDinterwordspacing}{\spaceskip=0pt\relax}
\providecommand{\BIBentryALTinterwordstretchfactor}{4}
\providecommand{\BIBentryALTinterwordspacing}{\spaceskip=\fontdimen2\font plus
\BIBentryALTinterwordstretchfactor\fontdimen3\font minus \fontdimen4\font\relax}
\providecommand{\BIBforeignlanguage}[2]{{%
\expandafter\ifx\csname l@#1\endcsname\relax
\typeout{** WARNING: IEEEtran.bst: No hyphenation pattern has been}%
\typeout{** loaded for the language `#1'. Using the pattern for}%
\typeout{** the default language instead.}%
\else
\language=\csname l@#1\endcsname
\fi
#2}}
\providecommand{\BIBdecl}{\relax}
\BIBdecl

\bibitem{TARTAKOVSKY_BOOK}
A.~Tartakovsky, I.~Nikiforov, and M.~Basseville, \emph{Sequential {A}nalysis: Hypothesis {T}esting and {C}hangepoint {D}etection}.\hskip 1em plus 0.5em minus 0.4em\relax CRC Press, 2014.

\bibitem{POOR_BOOK}
H.~V. {Poor} and O.~Hadjiliadis, \emph{Quickest Detection}.\hskip 1em plus 0.5em minus 0.4em\relax Cambridge University Press, 2008.

\bibitem{PAGE_1954}
E.~S. Page, ``Continuous inspection schemes,'' \emph{Biometrika}, vol.~41, no. 1/2, pp. 100--115, 1954.

\bibitem{LORDEN_1971}
G.~Lorden, ``Procedures for reacting to a change in distribution,'' \emph{The Annals of Mathematical Statistics}, vol.~42, no.~6, pp. 1897--1908, 1971.

\bibitem{MOUSTAKIDES_1986}
G.~V. Moustakides, ``Optimal stopping times for detecting changes in distributions,'' \emph{The Annals of Statistics}, vol.~14, no.~4, pp. 1379--1387, 1986.

\bibitem{Roberts1966}
S.~W. Roberts, ``A comparison of some control chart procedures,'' \emph{Technometrics}, vol.~8, pp. 411--430, 1966.

\bibitem{Polunchenko}
A.~S. Polunchenko and A.~G. Tartakovsky, ``On optimality of the {Shiryaev-Roberts} procedure for detecting a change in distribution,'' \emph{The Annals of Statistics}, vol.~38, no.~6, pp. 3445--3457, 2010.

\bibitem{LAI_1998}
T.~L. Lai, ``Information bounds and quick detection of parameter changes in stochastic systems,'' \emph{IEEE Transactions on Information Theory}, vol.~44, no.~7, pp. 2917--2929, 1998.

\bibitem{wang_review}
H.~Wang and Y.~Xie, ``Sequential change-point detection: Computation versus statistical performance,'' \emph{WIREs Computational Statistics}, vol.~16, no.~1, p. e1628, 2024.

\bibitem{Romano2023}
G.~Romano, I.~A. Eckley, P.~Fearnhead, and G.~Rigaill, ``Fast online changepoint detection via functional pruning {CUSUM} statistics,'' \emph{Journal of Machine Learning Research}, vol.~24, pp. 1--36, 2023.

\bibitem{Ward2025}
K.~Ward, G.~Dilillo, I.~Eckley, and P.~Fearnhead, ``Poisson-focus: An efficient online method for detecting count bursts with application to gamma ray burst detection,'' \emph{Journal of the American Statistical Association}, vol. 120, pp. 7--19, 2025.

\bibitem{DeLucia22042026}
J.~DeLucia and H.~V. Poor, ``Sequential decision problems for marked poisson processes,'' \emph{Sequential Analysis}, 2026.

\bibitem{ROBBINS_1974}
H.~Robbins and D.~Siegmund, ``The expected sample size of some tests of power one,'' \emph{The Annals of Statistics}, vol.~2, no.~3, pp. 415 -- 436, 1974.

\bibitem{ROBBINS_1972}
------, ``A class of stopping rules for testing parametric hypotheses,'' in \emph{Proceedings of the Sixth Berkeley Symposium on Mathematical Statistics and Probability (Univ. California, Berkeley, CA, 1970/1971)}, vol.~4, 1972, pp. 37--41.

\bibitem{SPARKS_2000}
R.~S. Sparks, ``{CUSUM} charts for signalling varying location shifts,'' \emph{Journal of Quality Technology}, vol.~32, no.~2, pp. 157--171, 2000.

\bibitem{LORDEN_2005}
G.~Lorden and M.~Pollak, ``{Nonanticipating estimation applied to sequential analysis and changepoint detection},'' \emph{The Annals of Statistics}, vol.~33, no.~3, pp. 1422 -- 1454, 2005.

\bibitem{LORDEN_2008}
------, ``Sequential change-point detection procedures that are nearly optimal and computationally simple,'' \emph{Sequential Analysis}, vol.~27, no.~4, pp. 476--512, 2008.

\bibitem{TARTAKOVSKY_2006}
A.~G. Tartakovsky, B.~L. Rozovskii, R.~B. Blažek, and H.~Kim, ``Detection of intrusions in information systems by sequential change-point methods,'' \emph{Statistical Methodology}, vol.~3, no.~3, pp. 252--293, 2006.

\bibitem{CAO_2018}
Y.~Cao, L.~Xie, Y.~Xie, and H.~Xu, ``Sequential change-point detection via online convex optimization,'' \emph{Entropy}, vol.~20, no.~2, 2018.

\bibitem{XIE_2023}
L.~Xie, G.~V. Moustakides, and Y.~Xie, ``Window-limited {CUSUM} for sequential change detection,'' \emph{IEEE Transactions on Information Theory}, vol.~69, no.~9, pp. 5990--6005, 2023.

\bibitem{brown2008admissible}
L.~D. Brown, E.~I. George, and X.~Xu, ``Admissible predictive density estimation,'' \emph{The Annals of Statistics}, vol.~36, no.~3, pp. 1156--1170, 2008.

\bibitem{WANG_2015}
Y.~Wang and Y.~Mei, ``Large-scale multi-stream quickest change detection via shrinkage post-change estimation,'' \emph{IEEE Transactions on Information Theory}, vol.~61, no.~12, pp. 6926--6938, 2015.

\bibitem{halme2025quickest}
T.~Halme, V.~V. Veeravalli, and V.~Koivunen, ``Quickest change detection for multiple data streams using the {James-Stein} estimator,'' \emph{IEEE Transactions on Information Theory}, 2025.

\bibitem{shafer_betting}
G.~Shafer, ``Testing by betting: A strategy for statistical and scientific communication,'' \emph{Journal of the Royal Statistical Society: Series A (Statistics in Society)}, vol. 184, no.~2, pp. 407--431, 2021.

\bibitem{shin_ramdas_rinaldo_2024}
J.~Shin, A.~Ramdas, and A.~Rinaldo, ``E-detectors: A nonparametric framework for sequential change detection,'' \emph{The New England Journal of Statistics in Data Science}, vol.~2, no.~2, pp. 229--260, 2024.

\bibitem{Grunwald_MDL_tutorial}
P.~Grunwald, ``A tutorial introduction to the minimum description length principle,'' 2004.

\bibitem{AITCHINSON_Goodness}
J.~Aitchison, ``Goodness of prediction fit,'' \emph{Biometrika}, vol.~62, no.~3, pp. 547--554, 1975.

\bibitem{Herbester_fixed_share}
M.~Herbster and M.~K. Warmuth, ``Tracking the best expert,'' \emph{Machine {Learning}}, vol.~32, no.~2, pp. 151--178, 1998.

\bibitem{Cesa-Bianchi_Lugosi_2006}
N.~Cesa-Bianchi and G.~Lugosi, \emph{Prediction, Learning, and Games}.\hskip 1em plus 0.5em minus 0.4em\relax Cambridge University Press, 2006.

\bibitem{adamskiy}
D.~Adamskiy, W.~M. Koolen, A.~Chernov, and V.~Vovk, ``A closer look at adaptive regret,'' in \emph{International Conference on Algorithmic Learning Theory}.\hskip 1em plus 0.5em minus 0.4em\relax Springer, 2012, pp. 290--304.

\bibitem{waudbysmith2025universallogoptimalitygeneralclasses}
\BIBentryALTinterwordspacing
I.~Waudby-Smith, R.~Sandoval, and M.~I. Jordan, ``Universal log-optimality for general classes of e-processes and sequential hypothesis tests,'' 2025. [Online]. Available: \url{https://arxiv.org/abs/2504.02818}
\BIBentrySTDinterwordspacing

\bibitem{Liang2024QuickestEstimation}
Y.~Liang and V.~V. Veeravalli, ``{Quickest Change Detection with Post-Change Density Estimation},'' \emph{IEEE Transactions on Information Theory}, p.~1, 2024.

\bibitem{tsybakov2008nonparametric}
A.~B. Tsybakov, ``Introduction to nonparametric estimation,'' in \emph{Introduction to Nonparametric Estimation}.\hskip 1em plus 0.5em minus 0.4em\relax Springer New York, NY, 2008.

\bibitem{Barlow}
R.~E. Barlow, A.~W. Marshall, and F.~Proschan, ``{Properties of Probability Distributions with Monotone Hazard Rate},'' \emph{The Annals of Mathematical Statistics}, vol.~34, no.~2, pp. 375 -- 389, 1963.

\bibitem{Wald1947SequentialAnalysis}
A.~Wald, \emph{{Sequential Analysis}}.\hskip 1em plus 0.5em minus 0.4em\relax John Wiley {\&} Sons, Inc., 1947.

\bibitem{warner2024worst}
A.~Warner and G.~Fellouris, ``Worst-case misidentification control in sequential change diagnosis using the min-cusum,'' \emph{IEEE Transactions on Information Theory}, vol.~70, no.~11, pp. 8364--8377, 2024.

\bibitem{XIE_2013}
Y.~Xie and D.~Siegmund, ``{Sequential multi-sensor change-point detection},'' \emph{The Annals of Statistics}, vol.~41, no.~2, pp. 670 -- 692, 2013.

\bibitem{chan2017optimal}
H.~P. Chan, ``Optimal sequential detection in multi-stream data,'' \emph{The Annals of Statistics}, vol.~45, no.~6, pp. 2736--2763, 2017.

\bibitem{chen2022high}
Y.~Chen, T.~Wang, and R.~J. Samworth, ``High-dimensional, multiscale online changepoint detection,'' \emph{Journal of the Royal Statistical Society Series B: Statistical Methodology}, vol.~84, no.~1, pp. 234--266, 2022.

\bibitem{rockova2024adaptivebayesianpredictiveinference}
\BIBentryALTinterwordspacing
V.~Rockova, ``Adaptive bayesian predictive inference in high-dimensional regerssion,'' 2024. [Online]. Available: \url{https://arxiv.org/abs/2309.02369}
\BIBentrySTDinterwordspacing

\bibitem{Castillo_haystack}
I.~Castillo and A.~van~der Vaart, ``Needles and straw in a haystack: Posterior concentration for possibly sparse sequences,'' \emph{The Annals of Statistics}, vol.~40, no.~4, pp. 2069--2101, 2012.

\bibitem{ocd_package}
\BIBentryALTinterwordspacing
Y.~Chen, T.~Wang, and R.~J. Samworth, \emph{ocd: High-Dimensional Multiscale Online Changepoint Detection}, 2020, r package version 1.1. [Online]. Available: \url{https://CRAN.R-project.org/package=ocd}
\BIBentrySTDinterwordspacing

\bibitem{jun_2017}
K.-S. Jun, F.~Orabona, S.~Wright, and R.~Willett, ``Improved strongly adaptive online learning using coin betting,'' in \emph{Artificial Intelligence and Statistics}.\hskip 1em plus 0.5em minus 0.4em\relax PMLR, 2017, pp. 943--951.

\bibitem{Lorden1970OnBoundary}
G.~Lorden, ``{On Excess Over the Boundary},'' \emph{The Annals of Mathematical Statistics}, vol.~41, no.~2, pp. 520--527, 4 1970.

\bibitem{siegmund1985sequential}
D.~Siegmund, \emph{Sequential Analysis: Tests and Confidence Intervals}.\hskip 1em plus 0.5em minus 0.4em\relax Springer Science \& Business Media, 1985.

\bibitem{johnstone2005ebayesthresh}
I.~Johnstone and B.~W. Silverman, ``{EbayesThresh}: R programs for {E}mpirical {Bayes} thresholding,'' \emph{Journal of Statistical Software}, vol.~12, pp. 1--38, 2005.

\bibitem{petersen2008matrix}
K.~B. Petersen, M.~S. Pedersen \emph{et~al.}, ``The matrix cookbook,'' \emph{Technical University of Denmark}, vol.~7, no.~15, p. 510, 2008.

\bibitem{normallaplace_R}
\BIBentryALTinterwordspacing
D.~Scott, J.~S. Fu, and S.~Potter, \emph{NormalLaplace: The Normal Laplace Distribution}, 2025, r package version 0.3-2. [Online]. Available: \url{https://CRAN.R-project.org/package=NormalLaplace}
\BIBentrySTDinterwordspacing

\end{thebibliography}

\appendix
\section{Omitted Proofs} 
\subsection{Proof of Lemma~\ref{thm:arl} (ARL of Predictive Mixture CuSum)}\label{app:proof_of_arl}

The proof follows a standard argument commonly used in sequential change detection, see e.g. \cite[Lemma 8.2.1]{TARTAKOVSKY_BOOK}, but we provide all details for completeness. Consider a Shiryaev-Roberts type statistic
\begin{equation}
    V_n = (1+V_{n-1})e^{\lhat_n}, \quad V_0 = 0,
\end{equation}
and the stopping time $\tau_b = \inf\{n : V_n \geq e^b\}$. It can be shown that $e^{S_n} \leq V_n$, and therefore $T_b \geq \tau_b$. Assume $\Ex_\infty(T_b) < \infty$, otherwise the result is trivial. It then follows that $\Ex_\infty(\tau_b) < \infty$.
The key observation is that since $\hat p_n$ is $\mathcal F_{n-1}$-measurable and
$X_n\sim q$, we have
\begin{equation}\label{eq:unitmean}
\mathbb E_\infty[e^{\lhat_n}\mid \mathcal F_{n-1}]
= \int \frac{\phat_n(x)}{q(x)}q(x)\,dx
= \int \phat_n(x)\,dx
= 1
\end{equation}
and
\[
\mathbb E_\infty[V_n\mid \mathcal F_{n-1}]
=
(1+V_{n-1})\mathbb E_\infty[e^{\lhat_n}\mid \mathcal F_{n-1}]
=
1+V_{n-1}.
\]
Therefore, the process $M_n := V_n - n$
is a zero-mean martingale with respect to $\{\mathcal F_n\}$ under $\mathbb P_\infty$. Then by the optional stopping theorem $\Ex_\infty(M_{\tau_b}) = \Ex_\infty(V_{\tau_b} - \tau_b) = 0$. Since $V_{\tau_b}\geq e^b$ by definition of $\tau_b$, we have $\Ex_\infty(\tau_b) \geq e^b$. Since $T_b \geq \tau_b$, the claim follows.

\subsection{Proof of Theorem~\ref{thm:delay_thm}}\label{app:proof_of_delay}

In proving the result, we consider the statistic 
\begin{equation}
    U_n^{\nu} = \sum_{i=\nu+1}^{n} \lhat_i = \sum_{i = \nu+1}^n \log \frac{\phat_i(X_i)}{q(X_i)}, \quad n \geq \nu + 1, \quad U_1^\nu,\dots,U_\nu^\nu = 0,
\end{equation}
and the related stopping time $T^\nu = \inf\{n \geq 1 : U_n^\nu > b\}$.
Observe first that $U_n^\nu \leq S_n$ for all $n$ and $\nu$ by construction, and therefore $T_b \leq T^\nu$. Hence, it is sufficient to prove the delay upper bound for $T^\nu$. 

By Wald's identity \cite[Corollary 2.3.1]{TARTAKOVSKY_BOOK} and adding and subtracting terms, we have
\begin{align}
    \Ex_\nu&((T^\nu-\nu +1)^+\mid \calF_{\nu-1}) = \Ex_\nu(T^\nu-\nu +1\mid \calF_{\nu-1})
     \\
    &=\frac{1}{\KL{p}{q}} \Ex_\nu\left(\sum_{n=\nu}^{T^\nu} \log \frac{p(X_n)}{q(X_n)} \mid \calF_{\nu-1}\right) \\
    &= \frac{1}{\KL{p}{q}} \left[b + \Ex_\nu\left(\sum_{n=\nu}^{T^\nu} \log \frac{p(X_n)}{q(X_n)} - U_T^\nu \mid  \calF_{\nu-1}\right) + \Ex_\nu\left(U_T^\nu -b \mid  \calF_{\nu-1}\right)\right]\\
    &= \frac{1}{\KL{p}{q}} \Big[b + \Ex_\nu\left(\log \frac{p(X_\nu)}{q(X_\nu)}\right) +  \Ex_\nu\left(\sum_{n = \nu+1}^{T^\nu} \log\frac{p(X_n)}{\phat_n(X_n)} \mid \calF_{\nu-1}\right) + \Ex_\nu\left(U^\nu_{T^\nu} -b\mid \calF_{\nu-1}\right)\Big]\\
    &= 1 + \frac{1}{\KL{p}{q}} \Big[b 
    + \underbrace{\Ex_\nu\left(\sum_{n = \nu+1}^{T^\nu} \log\frac{p(X_n)}{\phat_n(X_n)} \mid \calF_{\nu-1}\right)}_{(*)} 
    + \underbrace{\Ex_\nu\left(U^\nu_{T^\nu} -b\mid \calF_{\nu-1}\right)}_{(**)}\Big].\label{eq:wald_decomp}
\end{align}

Term $(*)$ relates to the cumulative estimation error of the predictive distribution $\phat_n$, while $(**)$ corresponds to the overshoot of the test statistic over the threshold $b$. An upper bound for the overshoot $(**)$ derived later in Lemma~\ref{lemma:overshoot}. 

Term $(*)$ can be further decomposed as follows: For $m \geq 1$, let 
    \begin{equation}
        \tilde w(m) = \max\{w \in \calW : w \leq m\} = \min\{2^{\lfloor\log_2 m\rfloor}, W_\text{max}\}
    \end{equation}
     denote the longest window less than or equal to $m$ in the set of windows $\calW$. Then when $n \geq \nu+1$, $(X_{\tilde w(n - \nu)},\dots,X_{n-1})$ contains only post-change observations. We can split the total estimation error into two components
    \begin{equation}
        \sum_{n = \nu+1}^{T^\nu} \log\frac{p(X_n)}{\phat_n(X_n)} = \sum_{n = \nu+1}^{T^\nu} \log\frac{p(X_n)}{\phat^{\tilde w(n-\nu)}_n(X_n)} +  \sum_{n = \nu + 1}^{T^\nu}  \log\frac{\phat^{\tilde w(n-\nu)}_n(X_n)}{\phat_n(X_n)},
    \end{equation}
    where the first term on the right-hand side represents the log-loss of an oracle algorithm that always uses only the predictor that has the largest ``clean'' window ${\tilde w(n-\nu)}$ at time $n$, and the second term is log-loss of the fixed share algorithm with respect to this oracle. Upper bounds for the two terms are derived below in Lemmas~\ref{lemma:oracle_approx_error} and \ref{lemma:fs_regret}.

    \begin{lemma}[KL-loss of Fixed Share]\label{lemma:fs_regret}
        \begin{equation}
            \begin{split}
                \Ex_\nu&\left(\sum_{n = \nu+1}^{T^\nu} \log\frac{\phat^{\tilde w(n-\nu)}_n(X_n)}{\phat_n(X_n)}\mid  \calF_{\nu-1}\right)
                \leq \frac{1}{b}\Ex_\nu\left(T^\nu-\nu+1 \mid\calF_{\nu-1}\right) + O((\log b)^2) \\
            \end{split}
        \end{equation}
    \end{lemma}
    \begin{proof}

    In proving the claim, we can utilize existing regret bounds for the Fixed Share algorithm. Each of the windows in the set $\calW$ represents an ``expert'', 
    and for any sequence of data and any interval $[t_1, t_2]$, $t_1 > 1$, it holds that \cite[Corollary 5]{adamskiy}
    \begin{equation}\label{eq:fs_adaptive_loss}
        \sum_{n=t_1}^{t_2} \log \frac{\phat_n^{(w)}(X_n)}{\phat_n(X_n)} \leq  \log (|\calW|-1) + \log\left(\frac{1}{\alpha}\right) + (t_2-t_1)\log\left(\frac{1}{1-\alpha}\right),
    \end{equation}
    where $|\calW|$ is the number of windows (or experts) and $\alpha$ is the fixed share parameter. In the literature on prediction with expert advice, this notion of regret is often called adaptive regret \cite{adamskiy}. Therefore, the total regret of the fixed share algorithm with respect to the oracle sequence $\{\tilde w(n-\nu)\}_{n = \nu +1}^{T^\nu}$ is upper bounded by
    \begin{align}
        \sum_{n = \nu + 1}^{T^\nu}  \log\frac{\phat^{\tilde w(n-\nu)}_n(X_n)}{\phat_n(X_n)} 
        & = \sum_{n = 1}^{T^\nu-\nu}  \log\frac{\phat^{\tilde w(n)}_{\nu + n}(X_{\nu + n})}{\phat_{\nu + n}(X_{\nu + n})} 
        = \sum_{w \in \calW}\sum_{n = 1}^{T^\nu-\nu} \ind{\tilde w(n) = w}  \log\frac{\phat^{(w)}_{\nu + n}(X_{\nu + n})}{\phat_{\nu + n}(X_{\nu + n})} \\
        &\leq |\calW| \log |\calW| + |\calW| \log\left(\frac {1}{\alpha}\right) + (T^\nu-\nu+1)\log\left(\frac {1}{1-\alpha}\right)\label{eq:fs_cumult_loss} \\
        &= \lceil\log_2 b\rceil \log(\lceil\log_2 b\rceil) + \lceil\log_2 b\rceil \log b +(T^\nu - \nu + 1)\frac{1}{b} + O(1)\label{eq:fs_cumult_loss_b} \\
        &\leq O(\log b\log\log b +\log^2b) + (T^\nu - \nu + 1)\frac{1}{b} \\
        &= (T^\nu - \nu + 1)\frac{1}{b} + O((\log b)^2)\label{eq:fs_cumult_loss_simple}
    \end{align}
    The above follows since $\{\tilde w(n-\nu)\}_{n = \nu +1}^{T^\nu}$ can be partitioned into $|\calW|$ segments where the window is fixed inside each segment, and utilizing~\eqref{eq:fs_adaptive_loss}. Since $\tilde{w}(n-\nu)$ is piecewise constant as $n$ grows and changes $|\calW|-1$ times, summing the interval regret bounds over the segments where $\tilde w(n-\nu)$ is constant yields the bound in \eqref{eq:fs_cumult_loss}.
    We note that upper bounds of the form~\eqref{eq:fs_cumult_loss} exist in the literature and are known as the switching regret \cite[eq. 6]{adamskiy}, \cite{Herbester_fixed_share}. The step from \eqref{eq:fs_adaptive_loss} to \eqref{eq:fs_cumult_loss} was presented to emphasize that \eqref{eq:fs_cumult_loss} holds irrespective of what the fixed share weights $\pi_\nu^{(w)}$ are at time $\nu$. Finally, \eqref{eq:fs_cumult_loss_b} is obtained by inserting $|\calW| = \lceil \log_2 b\rceil = O(\log b)$ and $\alpha = 1/b$ into \eqref{eq:fs_cumult_loss}. The claim follows by taking expectation of \eqref{eq:fs_cumult_loss_simple} on both sides.
\end{proof}

\begin{lemma}[KL-loss of oracle]\label{lemma:oracle_approx_error}
    \begin{equation}\label{eq:lemma2_statement}
        \Ex_\nu\left(\sum_{n = \nu+1}^{T^\nu} \log\frac{p(X_n)}{\phat^{{\tilde w(n-\nu)}}_n(X_n)} \mid \calF_{\nu-1} \right) \leq \frac{C_{1}}{b^\beta}\Ex_\nu(T^\nu - \nu +1 | \calF_{\nu-1}) + O(\zeta(b, \beta)),
    \end{equation}
    where
    \begin{equation}\label{eq:xi_def}
        \zeta(b, \beta) = \sum_{k=0}^{|\mathcal W|-1}2^{(1-\beta)k} = \begin{cases}
            O(\log b), \quad &\beta = 1 \\
            O(b^{1-\beta}), &0 < \beta < 1
        \end{cases}
    \end{equation}
\end{lemma}
    \begin{proof}
        Let us denote the terms in the sum on the left-hand side of \eqref{eq:lemma2_statement} by 
      \begin{equation}
            R_n = \log\frac{p(X_n)}{\phat^{{\tilde w(n-\nu)}}_n(X_n)}
      \end{equation} 
      for $n \geq \nu +1$.
        We have 
        \begin{equation}
            \Ex_\nu\left(\sum_{n = \nu+1}^{T^\nu} \log\frac{p(X_n)}{\phat^{{\tilde w(n-\nu)}}_n(X_n)} \mid \calF_{\nu-1} \right) = \Ex_\nu\left(\sum_{n = \nu+1}^{T^\nu + \Wmax} R_n \mid \calF_{\nu-1}\right) - \Ex_\nu\left(\sum_{n = T^\nu +  1}^{T^\nu + \Wmax} R_n \mid \calF_{\nu-1}\right). \label{eq:est_error_decomp}
        \end{equation}
        For the first term in the RHS of \eqref{eq:est_error_decomp}, we utilize a similar technique as in the proof of Theorem 1 of \cite{XIE_2023}. We make use of the fact that since the maximum window length used by the procedure is $\Wmax$, the sequence $\{R_n\}$ is $\Wmax$-dependent, i.e. $R_n$ is independent of $R_{n-\Wmax-1}$. We have
        \begin{align}
            \Ex_\nu\left(\sum_{n = \nu + 1}^{T^\nu + \Wmax}R_n \mid \calF_{\nu-1} \right) 
            &\overset{\text{}}{=} \Ex_\nu\left(\sum_{n = \nu + 1}^{\infty}R_n\ind{n \leq T^\nu + \Wmax} \mid \calF_{\nu-1} \right) \\
            &\overset{\text{(a)}}{=} \Ex_\nu\left(\sum_{n = \nu + 1}^{\infty}\Ex_\nu(R_n\ind{n \leq T^\nu + \Wmax}\mid \calF_{n-\Wmax -1}) \Big|  \calF_{\nu-1} \right) \\
            &\overset{\text{(b)}}{=} \Ex_\nu\left(\sum_{n=\nu+1}^{\infty}\Ex_\nu(R_n \mid\calF_{n-\Wmax-1})\ind{T^\nu + \Wmax \geq n} \Big| \calF_{\nu-1}\right) \\
            &\overset{\text{(c)}}{=} \Ex_\nu\left(\sum_{n=\nu + 1}^{T^\nu + \Wmax} \Ex_\nu(R_n) \Big | \calF_{\nu - 1}\right). \label{eq:R_sum}
        \end{align}
            where (a) is the tower property, (b) follows since $\{n \leq T^\nu + \Wmax\}$ is $\calF_{n-\Wmax -1}$-measurable and $R_n$ is independent of $\calF_{n-\Wmax-1}$, and (c) follows since $R_n$ is independent of $\calF_{n-\Wmax -1}$.

        The expectation inside the sum in \eqref{eq:R_sum} can be bounded using Assumption~\ref{assumption:kl_loss} by
        \begin{align}
            \Ex_\nu(R_{n})
            = \Ex_\nu\left(\log\frac{p(X_{n})}{\phat^{{(\tilde w(n-\nu))}}_{n}(X_{n})}\right)
            &=\Ex_1\left(\log\frac{p(X_{n})}{\phat^{(\tilde w(n-\nu))}_{n}(X_{n})}\right) =  \calR_1(\phat_{n}^{(\tilde w(n-\nu))}) \leq \frac{C_1}{\tilde w(n-\nu)^\beta} 
        \end{align}
        Then, using the definition of $\tilde{w}(n)$ and algebra, the sum of the expectations satisfies
        \begin{align}
            \sum_{n=\nu+1}^{T^\nu + \Wmax}\Ex_\nu(R_{n}) \leq \sum_{n=\nu+1}^{T^\nu + \Wmax}\frac{C_1}{\tilde w(n-\nu)^\beta}
            &= C_1\sum_{m=1}^{T^\nu +\Wmax -\nu}\frac{1}{\tilde w(m)^\beta} \notag\\
            &= C_1\left(\sum_{m=1}^{W_{\max}-1}\frac{1}{2^{\beta\lfloor \log_2 m\rfloor}}
            +\sum_{m=W_{\max}}^{T^\nu-\nu}\frac{1}{W_{\max}^\beta}\right) \notag\\
            &= C_1\left(\sum_{k=0}^{|\mathcal W|-1}2^k\,2^{-\beta k}
            +\frac{T^\nu-\nu+1-W_{\max}}{W_{\max}^\beta}\right) \notag\\
            &= C_1\left(\sum_{k=0}^{|\mathcal W|-1}2^{(1-\beta)k}
            +\frac{T^\nu-\nu+1-W_{\max}}{W_{\max}^\beta}\right) \\
            &\le C_{1}\left(\zeta(b,\beta)
            +\frac{T^\nu-\nu+1}{b^\beta}\right),\label{eq:beta_sum_final}
            \end{align}
            where $\zeta(b, \beta)$ was defined in \eqref{eq:xi_def} and in the last inequality we used the fact that $\Wmax = 2^{\lceil\log_2 b\rceil} \geq b$.
            Inserting \eqref{eq:beta_sum_final} into \eqref{eq:R_sum}, we obtain 
            \begin{equation}\label{eq:est_error_first_final}
                \Ex_\nu\left(\sum_{n = \nu + 1}^{T^\nu + \Wmax}R_n \mid \calF_{\nu-1} \right) \leq \frac{C_{1}}{b^\beta}\Ex_\nu(T^\nu - \nu +1 | \calF_{\nu-1}) + O(\zeta(b, \beta)).
            \end{equation}
            
            To prove \eqref{eq:lemma2_statement}, it remains to show that the quantity being subtracted in \eqref{eq:est_error_decomp} is non-negative. Since KL-divergence is non-negative, 
            $\Ex_\nu(R_n | \calF_{n-1}) \geq 0$ and we obtain the simple lower bound 
            \begin{align}
                \Ex_\nu\left(\sum_{n = T^\nu + 1}^{T^\nu + \Wmax} R_n ~\Big|~ \calF_{\nu-1}\right) &=  \Ex_\nu\left(\sum_{n = \nu + 1}^{\infty} R_n\ind{T^\nu + 1 \leq n \leq T^\nu + \Wmax} \mid \calF_{\nu-1}\right) \\
                & =\Ex_\nu\left(\sum_{n = \nu + 1}^{\infty} \Ex_\nu(R_n^\nu\mid\calF_{n-1})\ind{T^\nu + 1 \leq n \leq T^\nu + \Wmax} \mid \calF_{\nu-1}\right) \\
                &\geq 0. \label{eq:correction_final}
            \end{align}
            Equations \eqref{eq:est_error_first_final} and \eqref{eq:correction_final} together with \eqref{eq:est_error_decomp} prove \eqref{eq:lemma2_statement}.
            
            To confirm \eqref{eq:xi_def}, observe that when $\beta = 1$,
            \begin{equation}
                \zeta(b, \beta) = \sum_{k=0}^{|\mathcal W|-1}2^{(1-\beta)k} = \sum_{k=0}^{|\mathcal W|-1} = |\mathcal W| = \lceil\log_2 b\rceil = O(\log b)
            \end{equation}
            For \(0<\beta<1\), $\zeta(b, \beta)$ is a geometric series, so that
            \begin{align}
            \zeta(b, \beta) = \frac{2^{(1-\beta)|\mathcal W|}-1}{2^{1-\beta}-1} =
            \frac{W_{\max}^{\,1-\beta}-1}{2^{1-\beta}-1} = O(b^{1-\beta}),
            \end{align}
            since $\Wmax = 2^{\lceil\log_2 b\rceil} \leq 2b$.

        \end{proof}

        \begin{lemma}[Upper-bound on overshoot]\label{lemma:overshoot}
            \begin{equation}
                \Ex_\nu(U^\nu_{T^\nu} - b \mid  \calF_{\nu-1}) \leq C_2
            \end{equation}
        \end{lemma}
        \begin{proof}
            Write $r =  b - U^\nu_{T^\nu-1}$ as the gap between the threshold and the test statistic at time $T^\nu-1$. Then, by the definition of $T^\nu$, we have $\lhat_{T^\nu} > r$ on the event $\{T^\nu \geq \nu\}$, and therefore
            \begin{align*}
                \Ex_\nu(U^\nu_{T^\nu} -  b \mid  \calF_{\nu-1}) 
                &= \Ex_\nu[\Ex_\nu(U^\nu_{T^\nu} -  b \mid \calF_{T^\nu-1})] \\
                &= \Ex_\nu[\Ex_\nu(\lhat_{T^\nu} - (b - U^\nu_{T^\nu-1}) \mid \calF_{T^\nu-1}, T^\nu)] \\
                &=  \Ex_\nu[\Ex_\nu(\lhat_{T^\nu} -r \mid \calF_{T^\nu-1}, T^\nu)] \\
                &\leq \Ex_\nu\left[\sup_{r \geq 0}\Ex_\nu(\lhat_{T^\nu} -r \mid \calF_{T^\nu-1}, \lhat_{T^\nu} > r )\right]\\
                &\leq C_2.
            \end{align*}
            The last inequality follows from Assumption~\ref{assumption:overshoot_bound}.
        \end{proof}

            \subsection*{Proof of Theorem~\ref{thm:delay_thm}}
            Writing $x = \Ex_\nu((T^\nu-\nu +1)^+\mid \calF_{\nu-1})$ for brevity, we obtain from \eqref{eq:wald_decomp} and Lemmas \ref{lemma:fs_regret}, \ref{lemma:oracle_approx_error} and \ref{lemma:overshoot}, the inequality
            \begin{align}
                x &= 1 + \frac{1}{\KL{p}{q}} \left(b + \frac{x}{b} + O((\log b)^2)+ \frac{C_1}{b^\beta} x + O( \zeta(b,\beta)) + C_2 \right) \\
                 &\leq \frac{b}{\KL{p}{q}} + C\left(\frac{x}{b^\beta} +  (\log b)^2 + \zeta(b, \beta)\right),
            \end{align}
            for some constant $C >0$ independent of $b$.
            Rearranging,    
            \begin{equation}
                x \leq \left(1-\frac{C}{b^\beta}\right)^{-1}\left(\frac{b}{\KL{p}{q}} + (\log b)^2 + \zeta(b,\beta)\right),
            \end{equation}
            where $(1-Cb^{-\beta})^{-1} = 1 + O(b^{-\beta})$ for large $b$. Multiplying and dropping all terms asymptotically smaller than $\log b$, we obtain
            \begin{equation}\label{eq:delay_O_remainder}
                x \leq \frac{b}{\KL{p}{q}} + O(b^{1-\beta} + (\log b)^2 + \zeta(b,\beta)).
            \end{equation}
            Recall from \eqref{eq:xi_def} that $\zeta(b,\beta) = O(\log b)$ if $\beta = 1$, and $O(b^{1-\beta})$ otherwise. Therefore, if $\beta = 1$, the dominant remainder term in \eqref{eq:delay_O_remainder} is $O((\log b)^2)$, and if $0 < \beta < 1$, the dominant term is $b^{1-\beta}$, since $b^{1-\beta} >> (\log b)^2$ as $b \to \infty$. Finally, from Lemma~\ref{thm:arl}, we can substitute the threshold $b = \log \gamma$, and the proof is complete.       %

\section{Discussion of assumptions on the overshoot}\label{appendix:overshoot_discussion}

We note that Assumption~\ref{assumption:overshoot_bound} used for bounding the overshoot 
\begin{equation}
    \varkappa(b) = \Ex_\nu^P(S_{T_b} - b )
\end{equation}
may not be strictly necessary for obtaining the asymptotic rates of Theorem~\ref{thm:delay_thm}. Using the assumption, we obtain in Lemma~\ref{lemma:overshoot} in Appendix~A that $\varkappa(b)$ is bounded by a constant $C_2$ for all values of $b$. In the asymptotic regime $b \to\infty$ studied in the theorem, it could be possible to obtain a constant limiting overshoot $\lim_{b \to \infty}\varkappa(b)$ with weaker assumptions using the following reasoning. 
The test-statistic increments can be written as
\begin{equation}\label{eq:increment_decom}
    \lhat_n = \log\frac{\phat_n(X_n)}{q(X_n)}=\log\frac{p(X_n)}{q(X_n)}-\log \frac{p(X_n)}{\phat_n(X_n)} =: Z_n + Y^{(b)}_n,
\end{equation}
where $Z_n$ are i.i.d. random variables corresponding to the true likelihood ratios with mean $\KL{p}{q}$ after the change, and $Y^{(b)}_n$ represents a perturbation due to estimation error of $p$. In $Y_n^{(b)}$, we explicitly denote that the perturbation depends on the threshold $b$, since the design parameters $\calW$ and $\alpha = 1/b$ used by the PM-CuSum algorithm are chosen as a function of $b$. Suppose for simplicity that the change-point $\nu = 1$. Write 
\begin{equation}\label{eq:cumul_error_xi}
    \xi_n^{(b)} = \sum_{j=1}^n Y^{(b)}_j =\sum_{j=1}^n\log \frac{p(X_j)}{\sum_{w\in\calW}\pi_j^{(w)}\phat_j^{(w)}(X_j)}
\end{equation}
as the cumulative perturbation up to time $n$. Then, for large values of $n$, the proposed test statistic is well approximated by a random walk minus a perturbation term
\begin{equation}
    S_n \approx \sum_{j=1}^n \lhat_j=\sum_{j=1}^nZ_n - \xi^{(b)}_n,
\end{equation}
since $S_n$ will leave $0$ with probability one, after which the max-operation in the usual update of $S_n$ does not play a role. Intuitively, for large $n$ and $b$, the PM-CuSum algorithm is able to utilize long estimation windows such that $\phat_n$ well approximates $p$, and hence $\xi_n^{(b)}$ should grow slowly for large $n$. Indeed, most of the effort in the proof of Theorem~\ref{thm:delay_thm} is showing that that $\xi_n^{(b)}$ grows slow enough to not affect the detection delay too much. In fact, from the intermediate steps in the proof it can be deduced that for $n = O(b)$ 
\begin{equation}\label{eq:log_growth_xi}
    \Ex_1(\xi_n^{(b)}) = O((\log b)^2), \quad b  \to\infty.
\end{equation}
The effect of the perturbation $\xi_n^{(b)}$ becomes negligible compared to the random walk for large $n$ and $b$. It is therefore intuitive that the average overshoot of $S_{T_b}$ over a large threshold should be close to the average overshoot of just the random walk, which can be bounded by a constant assuming only a finite second moment for $Z_n$ using Lorden's inequality \cite{Lorden1970OnBoundary}. This is indeed the case if $\xi_n^{(b)}$ satisfies the conditions of a \emph{slowly changing} sequence \cite[pp. 50]{TARTAKOVSKY_BOOK}, \cite[pp. 190]{siegmund1985sequential}. The two conditions are
\begin{equation}\label{eq:sv1}
    \frac{1}{n}\max_{1\leq k\leq n}|\xi_k^{(b)}| \longrightarrow 0 \quad \text{in } P\text{-probability},
\end{equation}
and that for every $\epsilon > 0$ there are $n^*$ and $\delta > 0$ such that
\begin{equation}\label{eq:sv2}
    P\left(\max_{1\leq k \leq n\delta}|\xi^{(b)}_{n+k}-\xi^{(b)}_n| > \epsilon\right) < \epsilon \quad \text{for all } n \geq n^*.
\end{equation}
For the purposes of this paper, the parameter $b$ must also grow to infinity at a sufficient rate in \eqref{eq:sv1} and \eqref{eq:sv2}. 

Condition \eqref{eq:sv1} could be confirmed from \eqref{eq:log_growth_xi} with some second moment assumptions, but unfortunately we have been unable to rigorously prove \eqref{eq:sv2} for the considered $\xi^{(b)}_n$. Condition \eqref{eq:sv2} states that eventually the perturbation sequence $\xi_n^{(b)}$ has jumps with small probability. While there exist bounds for the regret of the Fixed Share algorithm in any interval, such the one used in \eqref{eq:fs_adaptive_loss}, these bounds hold for \emph{any} (even adversarial) data $X_n$, and are therefore too conservative for proving \eqref{eq:sv2}. In the considered setup, the data is i.i.d. and it is intuitive that the FS algorithm will perform much better in the limit $n\to\infty$ than the worst-case bounds suggest. However, obtaining precise estimates for the asymptotic behavior of $\xi_n^{(b)}$ in \eqref{eq:cumul_error_xi} is tedious due to the dependencies caused by the FS weight update $\pi_n^{(w)}$, and the overlapping sliding windows of different lengths. 

\section{Details of proposed sparse Gaussian predictor}\label{appendix:gauss_shift_predictives}
In this section, we provide omitted details of the proposed sparse Gaussian predictor used in Section~\ref{sec:gaus_shift}. The derivations make use of some ideas presented in \cite{johnstone2005ebayesthresh}, but we present all steps for completeness. Moreover, \cite{johnstone2005ebayesthresh} considers only estimation of $\btheta$, while our focus is on prediction of the next observation.
\subsection*{Closed-form evaluation of the Laplace-slab predictive density}
In this section, we derive the Laplace-slab predictive term $u(x|z)$ \eqref{eq:u_def} used in Section~\ref{sec:gaus_shift}:
\begin{align}\label{eq:u_appendix}
u(x\mid z) &= p(X_n^{(j)}=x|\bar X^{(j)}=z, A^{(j)}=1)= \int \calN(x;\vartheta,\sigma^2)\frac{\calN(z;\vartheta,\sigma^2/w)g_\lambda(\vartheta)}
{m_{1}(z)}d\vartheta \\
&= \frac{1}{m_1(z)} \underbrace{\int \calN(x;\vartheta,\sigma^2)\calN(z;\vartheta,\sigma^2/w)g_\lambda(\vartheta)d\vartheta}_I\label{eq:u_integral}
\end{align}
where $m_1(z)$ is given from \eqref{eq:m_dense} as
\begin{equation}\label{eq:m1_appendix}
m_{1}(z) =p(\bar X^{(j)} = z|A^{(j)}=1)=\int \calN(z;\vartheta,\sigma^2/w)g_\lambda(\vartheta)\,d\vartheta 
\end{equation}
and $g_\lambda$ is the Laplace density \eqref{eq:lapl_dist}.

Using standard Gaussian identities, see e.g. \cite[Sec. 8.1.8]{petersen2008matrix}, it holds that
\begin{equation}
    \calN\left(x; \vartheta, \sigma^2\right)\calN\left(z;\vartheta, \sigma^2/w\right) = \calN\left(x-z; 0, \sigma^2+\frac{\sigma^2}{w}\right)\calN\left(\vartheta;\mu, \tau^2\right),
\end{equation}
where $\mu = (x + wz)/(1+w)$, $\tau^2 = \sigma^2/(1+w)$. Therefore, we obtain for the integral in \eqref{eq:u_integral}
\begin{align}
   I &= \calN\left(x-z; 0, \sigma^2+\frac{\sigma^2}{w}\right) \int \calN\left(\vartheta;\mu, \tau^2\right)g_\lambda(\vartheta)d\vartheta \\
   &= \calN\left(x-z; 0, \sigma^2+\frac{\sigma^2}{w}\right) \int \calN\left(\mu-\vartheta;0, \tau^2\right)g_\lambda(\vartheta)d\vartheta,
\end{align}
where we notice that the integral is a convolution between a normal distribution and a Laplace distribution. That is, if $Y \sim \calN(0, \tau^2)$ and $\vartheta \sim \text{Laplace}(\lambda)$ are independent, then the pdf of $Y+\vartheta$ is given by
\begin{equation}\label{eq:h_def}
    h(y; \tau^2, \lambda) := \int \calN(y-\vartheta;0, \tau^2)\frac{\lambda}{2}\exp(-\lambda |\vartheta|)d\vartheta,
\end{equation}
and then 
\begin{equation}
    I = \calN\left(x-z; 0, \sigma^2+\frac{\sigma^2}{w}\right)h(\mu;\tau^2,\lambda).
\end{equation}
The convolution can be computed in closed form, and is implemented in e.g. the R package \texttt{NormalLaplace} \cite{normallaplace_R}. 
For the normalizing term in \eqref{eq:u_integral}, we have from \eqref{eq:m1_appendix}
\begin{equation}\label{eq:m1_closed}
    m_1(z) = \int \calN(z;\vartheta,\sigma^2/w)g_\lambda(\vartheta)\,d\vartheta = h(z;\sigma^2/w,\lambda).
\end{equation}
Combining the expressions for $I$ and $m_1$, we obtain
\begin{equation}
    u(X_n^{(j)} | \bar{X}^{(j)}) = \phi\left(X_n^{(j)}-\bar{X}^{(j)}; 0, \sigma^2+\frac{\sigma^2}{w}\right)\frac{h(\mu; \tau^2, \lambda)}{h(\bar{X}^{(j)}; \sigma^2/w, \lambda)},
\end{equation}
where 
\begin{equation}
    \mu = \frac{w\bar{X}^{(j)} + X_n^{(j)}}{w+1}, \quad \tau^2 = \frac{\sigma^2}{w+1}.
\end{equation}

\subsection*{Posterior probability of Laplace-component}

Next, we derive a closed-form expression for the posterior probability that coordinate $j$ is affected by the change given the observed data $\bar X^{(j)}$, that is
\begin{equation}
    \rho(z)=\mathbb P(A^{(j)}=1\mid \bar X^{(j)}=z)
=
\frac{\eta m_{1}(z)}
{(1-\eta)m_{0}(z)+\eta m_{1}(z)},
\end{equation}
as derived in \eqref{eq:post_prob}. Using \eqref{eq:m1_closed}, we get
\begin{equation}
    \rho(z) = \frac{\eta \, h(z;\sigma^2/w, \lambda)}{(1-\eta)\calN(z;0,\sigma^2/w) +\eta \, h(z;\sigma^2/w, \lambda)}.
\end{equation}
Writing 
\begin{equation}\label{eq:kappa_def}
    \kappa(z) = \frac{h(z; \sigma^2/w, \lambda)}{\calN(z;0,\sigma^2/w)} -1,
\end{equation}
the posterior probability $\rho(z)$ simplifies to
\begin{equation}
    \rho(z) = \frac{1 + \kappa(z)}{1/\eta + \kappa(z)}.
\end{equation}

\subsection*{Solving for $\eta$}

Finally, using the expression for the marginal likelihood $m_1(z)$ derived in \eqref{eq:m1_closed}, the maximum marginal likelihood estimate for $\eta$ is given by
\begin{align}
    \hat\eta &= \underset{\eta}{\arg \max} \prod_{j=1}^k\left[(1-\eta)m_0(\bar X^{(j)}) + \eta \, m_1(\bar X^{(j)})\right] \\
    &= \underset{\eta}{\arg \max} \underbrace{\sum_{j=1}^k \log \left[ (1-\eta)\calN(\bar X^{(j)}; 0, \sigma^2/w) + \eta \, h(\bar X^{(j)}; \sigma^2/w, \lambda) \right]}_{l(\eta)} \\
    &= \underset{\eta}{\arg \max} \, l(\eta).
\end{align} 
The derivative of the log-likelihood is given by
\begin{align}
    l'(\eta) &= \sum_{j=1}^k \frac{h(\bar X^{(j)}; \sigma^2/w, \lambda) - \calN(\bar{X}^{(j)}; 0, \sigma^2/w)}{\eta \, h(\bar X^{(j)}; \sigma^2/w, \lambda) + (1-\eta)\calN(\bar{X}^{(j)}; 0, \sigma^2/w)} \\
    &= \sum_{j=1}^{k} \frac{\kappa(\bar X^{(j)})}{1+\eta \kappa(\bar X^{(j)})},
\end{align}
where $\kappa(\bar X^{(j)})$ was defined in \eqref{eq:kappa_def}.
Since $l'(\eta)$ is decreasing in $\eta$ on $[0,1]$, the maximum marginal likelihood estimate is attained at the root of $l'(\eta)$ if it exists, and at a boundary otherwise. The root can be found using binary search, for example.

\end{document}